\documentclass[11pt]{amsart}
\usepackage{amscd}
\usepackage{amsfonts}
\usepackage{amsmath,amsthm,hyperref}
\usepackage[all]{xy}
\usepackage{amsmath,amsthm,hyperref}
\usepackage{amsmath,amssymb,amsthm,latexsym}
\usepackage{xcolor}
\usepackage{amsmath}
\usepackage{bm}
\usepackage{dsfont}
\usepackage{tikz-cd}

\usepackage{enumerate}
\numberwithin{equation}{section}
\newtheorem{theorem}{Theorem}[section]
\newtheorem{proposition}[theorem]{Proposition}
\newtheorem{definition}[theorem]{Definition}
\newtheorem{corollary}[theorem]{Corollary}
\newtheorem{lemma}[theorem]{Lemma}

\theoremstyle{remark}
\newtheorem{remark}[theorem]{Remark}
\newtheorem{example}[theorem]{\bf Example}
\newcommand{\R}{\mathbb{R}}
\newcommand{\C}{\mathbb{C}}
\newcommand{\D}{\mathbb{D}}
\newcommand{\E}{\mathbb{E}}
\newcommand{\T}{\mathbb{T}}
\newcommand{\dd}{\mathrm{d}}
\newcommand{\FC}{\mathcal{C}}

\newcommand{\violet}{\textcolor{violet}}

\newcommand{\Str}{\mathbb{S}}

\begin{document}
\title[Harmonic maps of finite uniton type ]{\bf{Harmonic maps of finite uniton type into inner symmetric spaces}}
\author{Josef F. Dorfmeister, Peng Wang}
\maketitle

\begin{abstract}
In this paper, we develop a loop group description of harmonic maps $\mathcal{F}: M \rightarrow G/K$ ``of finite uniton type",  from a Riemann surface $M$ into  inner symmetric spaces of compact or non-compact type.  {This develops work of Uhlenbeck, Segal, and Burstall-Guest to non-compact inner symmetric spaces.} To be more concrete, we prove that every harmonic map of finite uniton type from any  Riemann surface into any  {compact or non-compact} inner symmetric space has a normalized potential taking values in some nilpotent Lie subalgebra, as well as a normalized frame with initial condition identity. This provides a straightforward way to construct all such harmonic maps. We also illustrate the above results exclusively by Willmore surfaces, since this  problem is motivated by the study of Willmore two--spheres in spheres.
 \vspace{2mm}

\end{abstract}

{\bf Keywords:}  harmonic maps of finite uniton type; non-compact inner symmetric spaces; normalized potential; Willmore surfaces.\\

MSC(2010): 58E20; 53C43;  53A30; 53C35

\tableofcontents
\section{Introduction}

 Harmonic maps from Riemann surfaces into symmetric spaces arise naturally in geometry and mathematical physics and hence became important objects in several mathematical fields, including  the study of  minimal surfaces, CMC surfaces, Willmore surfaces and related integrable systems. For harmonic maps into compact symmetric spaces or compact Lie groups, one of the most foundational and important papers is the  {description} of all harmonic two spheres into $U(n)$ by Uhlenbeck \cite{Uh}. It was shown that harmonic two-spheres satisfy a very  restrictive condition. Uhlenbeck coined the expression  `` finite uniton number'' for this property. Since the uniton number is an integer, one also obtains this way a subdivision of harmonic maps into  $U(n)$.

 Uhlenbeck's work was generalized in a very elegant way to harmonic two spheres in all compact semi-simple Lie groups  {and into all compact inner symmetric spaces} by Burstall and Guest in \cite{BuGu}. Using  Morse theory for loop groups in the spirit of Segal's work, they showed that  {harmonic maps of finite uniton number} in compact Lie groups can be related to  {meromorphic } maps into  (finite dimensional)  nilpotent  Lie algebras. They also provided a concrete method to find all such nilpotent Lie algebras. Finally, via the Cartan embedding, harmonic maps into compact inner symmetric spaces are considered as harmonic maps into Lie groups satisfying some algebraic ``twisting'' conditions. Therefore the theory of Burstall and Guest provides a description of  {harmonic maps of finite uniton number} into compact Lie groups and compact inner symmetric spaces, in a way which is not only theoretically satisfying, but can also be implemented well for concrete computations \cite{Gu2002}.

 {Let us take a somewhat closer look at the paper \cite{BuGu}. The title and the introduction of loc.cit. only refer to harmonic maps $\mathcal F: S^2 \rightarrow \Str$, where $\Str$ is a compact  inner symmetric space. A large part of the body of the paper, however, deals with harmonic maps  $\mathcal F: M \rightarrow \Str$, where $\Str$ is as above and $M$ is an arbitrary, connected, compact or non-compact, Riemann surface satisfying the following conditions:
 \begin{itemize}
 \item There exists an extended solution $\Phi(z, \bar z, \lambda): M \rightarrow \Str$, $\lambda \in S^1$,
 such that
\[  \hbox{$\Phi$ has a finite uniton number $k \geq 0$.}\]
 \end{itemize}
 Such harmonic maps are said to be of finite uniton number $k$ (See Section 1 of \cite{BuGu} for more details).}

  {In the present paper we use the DPW method to investigate harmonic maps into non-compact  inner symmetric spaces. Therefore we consider extended frames, not extended solutions.}   These frames are defined on the universal cover $\tilde{M}$ of $M$.
   { (One can  include, with some caveat, the case $M = S^2$,  see Section 3 of \cite{DoWa12} and
  $(2)$ of Remark \ref{sphere}).}
   {The two conditions above defining finite uniton  number harmonic maps  translate into the following two properties of  extended frames}  {(see Proposition \ref{typeequivnumber} and Proposition \ref{prop-fut}) :}
  \begin{theorem}
  {Let  $\mathcal F: M\rightarrow \Str$ be a harmonic map from a connected Riemann surface $M$ into  a compact, inner symmetric space.
 Then $\mathcal F$ has a finite uniton number $($$0 \leq k$ for some integer $k$$)$ if and only if}
 \begin{enumerate}
 \item  { There exists an extended frame $F : \tilde{M} \rightarrow \Lambda G_\sigma$
 for $\mathcal F$ which has  trivial monodromy relative to the action of $\pi_1(M)$ on $\tilde{M}$.}
\item  {There exists some frame  $F : \tilde{M} \rightarrow \Lambda G_\sigma$ for $\mathcal F$,  whose Fourier expansion relative to $\lambda$ is a Laurent polynomial.}
 \end{enumerate}
Here $ \Lambda G_\sigma$  denotes the twisted loop group associated to $\Str$.
  \end{theorem}

   Note that property $(1)$ for some extended frame is equivalent to property $(1)$ for all extended frames.  Harmonic maps $\mathcal F$, as well as the corresponding extended frames, are said  to be of {\bf finite uniton type} if the properties $(1)$ and $(2)$ are satisfied.
   We will apply the notion of ``finite uniton type" analogously, if $\Str$ is a non-compact symmetric space (see Definition \ref{def-uni}). A discussion of the properties $(1)$ and $(2)$  will be given in Section 3.   Both properties together (i.e. the case of finite uniton type harmonic maps) will be investigated in Section 4. Moreover, in
  Proposition \ref{typeequivnumber} we will show that for a compact Riemann surface $M$ harmonic maps of finite uniton type are  in a bijective relation with  finite uniton number harmonic maps  in the sense of Uhlenbeck \cite{Uh} (see also \cite{BuGu}).

 {While in the literature  primarily harmonic maps $\mathcal F: M\rightarrow \Str$ were considered, where $\Str$ is a compact symmetric space,  in the theory of Willmore surfaces in $S^n$, and many other surface classes, one has to deal with ``Gauss type maps" which are harmonic maps into non-compact symmetric spaces. }
It is the general goal of this paper to generalize results of \cite{BuGu} to harmonic maps of finite uniton type into a non-compact inner symmetric space. In particular, we want to describe simple potentials (in the sense of  {the  DPW method}) which generate such surfaces. For a compact inner symmetric target space this task has been carried out fairly explicitly in \cite{BuGu}, see subsections \ref{f.u.alaBuGu} and \ref{BuGu<->DPW}  below for a description in our notation.

For the case of a non-compact symmetric target space no such description is known yet. Thus one looks for an approach which permits to apply the work of \cite{BuGu} in such a way that one can also find simple potentials for the case of a  non-compact inner symmetric target space of a harmonic map.
  In \cite{DoWa13}, when dealing with harmonic maps into $SO^+(1,n+3)/SO(1,3)\times SO(n)$, the authors found a simple way to relate harmonic maps into a  non-compact inner symmetric space $G/K$  to harmonic maps into the  {dual compact} inner symmetric space $U/(U\cap K^{\C})$ dual to $G/K$.
  These two harmonic maps have a simple, but very important  relationship:  they share the  {same} meromorphic extended  framing and the normalized potential (see Theorem 1.1 of \cite{DoWa13} and Theorem \ref{thm-noncompact} in this paper). Here the normalized  extended framing and the normalized potential are meromorphic  data related to a harmonic map in terms of the language of the DPW method \cite{DPW}, which is a generalized Weierstrass type representation for harmonic maps into symmetric spaces.

   Interpreting the work of Burstall and Guest, one will see that for harmonic maps of finite uniton type into a compact symmetric space, their work considers normalized potentials which take values  in some (fixed) nilpotent Lie subalgebra (of the originally given finite dimensional complex Lie algebra) and their extended meromorphic frames take values in the  loop group of the corresponding unipotent Lie subgroup.   Therefore, for each fixed value of the loop parameter these meromorphic extended frames
take values in  the  finite dimensional unipotent Lie group mentioned   above
(see  Theorem \ref{thm-finite-uniton2}  {in this paper}, or  Theorem 1.11 \cite{Gu2002}).
 It thus turns out that by combining the dualization procedure of  \cite{DoWa13} with the  {grouping by the uniton number} of finite uniton number harmonic maps of \cite{BuGu}, one is able to characterize all harmonic maps of finite uniton type  into non-compact inner symmetric spaces by characterizing all the normalized extended  frames and the normalized potentials of harmonic maps of finite uniton type into compact inner symmetric spaces, which, according to the theory of Burstall and Guest, can be
  {described precisely.} Simply speaking, the case of a harmonic map into a non-compact inner symmetric space $G/K$ comes exactly from the case of a harmonic map into the compact dual inner symmetric space $U/(U\cap K^{\C})$ by choosing the same normalized potential for both harmonic maps, but using the two different real forms $G$ and $U$ of $G^\C$ for the loop group construction of the corresponding harmonic maps  {(see Theorem \ref{thm-finite-uniton-n-com})}.

From a technical point of view  it is important to observe (as pointed out above already) that in \cite{BuGu} the construction of harmonic maps uses  ``extended solutions", while the loop group method  uses extended frames. It is therefore a priori difficult to relate these two construction schemes to each other. {For the convenience of the reader and to fix notation we start a comparison of these two methods by recalling the relationship
between the extended solutions and the extended frames associated to a  harmonic map into a symmetric space. Then we introduce  the main results of Burstall and Guest on harmonic maps of finite uniton type \cite{BuGu}, as well as a description of their work in terms of {\em normalized potentials}, some of which has appeared in \cite{BuGu} and \cite{Gu2002}. Applying the duality theorem \cite{DoWa13}, we obtain the Burstall-Guest theory for the cases of non-compact inner symmetric spaces.}
 {Both theories occurring in this paper consider group valued (actually ``matrix valued") function systems satisfying certain (partial) differential equations in the space variable with dependence on some ``loop parameter".
The ``extended solution approach" only fixes the solutions for two values of the loop parameter, while the ``extended frame approach = DPW method" fixes values of these frames for each loop parameter at a fixed basepoint in the domain of definition. As a simple consequence, the DPW method works with unique solutions and essentially bijective relations between potentials and harmonic maps. For simplicity we frequently say : the frame $F$ has initial condition $F(z_0, \bar z_0,\lambda)$, when we should spell out more explicitly : the initial condition of the frame $F(z, \bar z,\lambda)$ at the base point $z_0$ is $F(z_0, \bar z_0,\lambda)$.}

 {Finally, here comes the problem of initial conditions in the study of harmonic maps of  finite uniton type:  This turns out be crucial in the cases of non-compact inner symmetric spaces, although it is not a big problem in the compact cases. It comes from the fact that the Iwasawa decomposition for compact loop groups  is global while the Iwasawa decomposition for non-compact ones  is local (See for example Theorem \ref{thm-decomposition}). Also note that the freedom of initial conditions  also  is equivalent to the freedom of special dressing actions. In this sense, a fixed initial condition will simplify the classification of harmonic maps of finite uniton type further, which makes it more simple to derive geometric properties of harmonic maps via normalized potentials. A standard example is the description of minimal surfaces in $\R^n$ via potentials in \cite{Wang-2}.}  {In Theorem \ref{thm-finite-uniton-in} we can show that the initial condition of the meromorphic extended  frame can be set without loss of generality to be identity.}

Our main motivation for the study of such harmonic maps is to provide the background for a detailed study of a wide variety of different types  of Willmore surfaces in $S^n$, surfaces of compact or non-compact type. As an application of this paper a rough classification of Willmore two-spheres  (whose conformal Gauss maps  take values in  the non-compact symmetric space $SO^+(1,n+3)/SO(1,3)\times SO(n)$) has been worked out in \cite{Wang-1}. For the convenience of the reader we  include the main result of \cite{Wang-1} by presenting its coarse classification of Willmore two-spheres in $S^{2n}$, in terms of the normalized potentials of their conformal Gauss maps. Moreover, we also present a new Willmore two-sphere constructed by using \cite{DoWa11,DoWa12} and the results of this paper. This example also responds to an open problem posed by  Ejiri in 1988 \cite{Ejiri1988}. \vspace{2mm }

This paper is organized as follows. In Section 2 we review the basic results of the loop group theory for harmonic maps.
In Section 3, we provide a detailed description of harmonic maps of finite uniton type.
Several equivalent definitions  are given for such maps.
Moreover, we also discuss briefly the monodromy and dressing actions for harmonic maps of finite uniton type.  {In Section 4 we first introduce the main results of Burstall and Guest on harmonic maps of finite uniton type \cite{BuGu}, as well as a description of their work in terms of {\em normalized potentials}, most of which have appeared in \cite{BuGu} and \cite{Gu2002}. Then we prove that for harmonic maps of finite uniton numbers into compact inner symmetric spaces, the initial condition of the extended frame and the extended meromorphic frame can be identity at some chosen base point in $M$.
 Finally in Section 5, we first apply the above result to get the same description of harmonic maps of finite uniton type into non-compact inner symmetric spaces. As an illustration, an outline of applications to the study of Willmore surfaces is listed.} \\

\section{Review of basic loop group theory}

 For any  inner involution $\sigma$ of a semi-simple  {real} Lie group $G$ the center of $G$ is contained in the connected component of the fixed point set of $\sigma$.
To begin with, we first recall some  notation. Let $G$ be a connected real semi-simple Lie group, compact or non-compact, represented as a matrix Lie group. Let $G/K$ be an inner symmetric space with the involution
$\sigma: G\rightarrow G$ such
that $G^{\sigma}\supset K\supset(G^{\sigma})^0$, where ``0" denotes ``identity component".
{\em  For the purposes of this paper the actual choice of $K$ will be of little importance. The reader may thus simply assume that $K = \hat{K} = G^\sigma$ holds.}
 In particular, we can assume without loss of generality that $G$ has trivial center.
 We will keep this assumption throughout this paper, except where we state the opposite.
Note that  {(the tangent bundle of) } $G/K$ carries a left-invariant non-degenerate symmetric bilinear form.
Let $\mathfrak{g}$ and $\mathfrak{k}$ denote the
Lie algebras of $G$ and $K$ respectively. The involution $\sigma$ induces
a decomposition of $\mathfrak{g}$ into eigenspaces,  the (generalized) Cartan decomposition
\[\mathfrak{g}=\mathfrak{k}\oplus\mathfrak{p},\hspace{5mm} \hbox{ with }\  [\mathfrak{k},\mathfrak{k}]\subset\mathfrak{k},
~~~ [\mathfrak{k},\mathfrak{p}]\subset\mathfrak{p}, ~~~
[\mathfrak{p},\mathfrak{p}]\subset\mathfrak{k}.\]
Let $\pi:G\rightarrow G/K$ denote the projection of $G$ onto $G/K$.

Now let  $\mathfrak{g^{\mathbb{C}}}$ be the complexification of $\mathfrak{g}$ and $G^{\mathbb{C}}$  the connected complex (matrix) Lie group with Lie algebra $\mathfrak{g^{\mathbb{C}}}$.
Let $\tau$ denote the complex anti-holomorphic involution
$g \rightarrow \bar{g}$, of $G^{\mathbb{C}}$. Then $G=Fix^{\tau}(G^{\C})^0$.
The inner involution $\sigma: G \rightarrow G $ commutes with the complex conjugation $\tau$ and extends to
the complexified Lie group $G^\C$, $\sigma: G^{\mathbb{C}}\rightarrow G^{\mathbb{C}}$. Let $K^{\mathbb{C}}\subset \hbox{Fix}^{\sigma}(G^{\mathbb{C}})$ denote the smallest complex subgroup of $G^{\C}$ containing $K$. Then the Lie algebra of $K^{\mathbb{C}}$ is
$\mathfrak{k^{\mathbb{C}}}$.

Occasionally we will also use another complex anti-linear involution, $\theta$, which commutes with $\sigma$ and $\tau$ and has as fixed point set
$ Fix^{\theta}(G^{\C})$,   a  maximal compact subgroup of
$G^{\C}$. For more details on the basic setting we refer to \cite{DoWa13}.

\begin{remark}
In this paper we only consider inner symmetric spaces. However, several of our results
also hold for arbitrary symmetric spaces. To keep the presentation of the paper
as simple as possible we will not consider  { the case of outer symmetric spaces in any detail in this paper. }
\end{remark}


\subsection{Harmonic maps into symmetric spaces}

Let $G/K$ be  {an inner symmetric} space as above and let $\mathcal{F}:M\rightarrow G/K$ be a harmonic map  {where $M$ is a connected Riemann surface.}
  {\bf In this paper we will always assume  that $\mathcal{F}$ is ``full"}   {in the following sense. We will mention the assumption "full" only where it seems to be particularly important.}

 {
\begin{definition} \label{deffull}
A harmonic map $\mathcal{F}:M\rightarrow G/K$ is called ``full"  if only $g = e$ fixes every element
of $\mathcal{F}(M)$.
That is,  if there exists $g\in G$ such that $g \mathcal{F}(p)=\mathcal{F}(p)$ for all
$p\in M$, then $g=e$.
\end{definition}}

 {Let $\tilde{\pi} : \tilde{M} \rightarrow M$ be the universal cover of $M$ and $z_0 \in \tilde{M}$ satisfying $\tilde{\pi}(z_0) = p_0$.
Then $\mathcal{F}$ has a natural lift $\tilde{\mathcal{F}}: \tilde{M} \rightarrow G/K$ satisfying
$\tilde{\mathcal{F}} = \mathcal{F} \circ \tilde{\pi}$ and obviously  $\tilde{\mathcal{F}}(z_0,\bar z_0) = eK$. Moreover, there exists a frame $F: \tilde{M} \rightarrow G$ such that $\tilde{\mathcal{F}}=\pi \circ F$ and $F(z_0,\bar z_0) = e.$}

Let $\alpha$ denote the Maurer-Cartan form of $F$. Then $\alpha$ satisfies the Maurer-Cartan equation
and we have
\begin{equation*}F^{-1}\mathrm{d} F= \alpha, \hspace{5mm} \mathrm{d} \alpha+\frac{1}{2}[\alpha\wedge\alpha]=0.
\end{equation*}
Decomposing $\alpha$ with respect to $\mathfrak{g}=\mathfrak{k}\oplus\mathfrak{p}$ we obtain
\[\alpha=\alpha_{ \mathfrak{k}  } +\alpha_{ \mathfrak{p} }, \
\alpha_{\mathfrak{k  }}\in \Gamma(\mathfrak{k}\otimes T^*M),\
\alpha_{ \mathfrak{p }}\in \Gamma(\mathfrak{p}\otimes T^*M).\] Moreover, considering the complexification $TM^{\mathbb{C}}=T'M\oplus T''M$, we decompose $\alpha_{\mathfrak{p}}$ further into the $(1,0)-$part $\alpha_{\mathfrak{p}}'$ and the $(0,1)-$part $\alpha_{\mathfrak{p}}''$. Set  \begin{equation} \label{alphalambda}
\alpha_{\lambda}=\lambda^{-1}\alpha_{\mathfrak{p}}'+\alpha_{\mathfrak{k}}+\lambda\alpha_{\mathfrak{p}}'', \hspace{5mm}  \lambda\in S^1.
\end{equation}

\begin{lemma} $($\cite{DPW}$)$ The map  $\mathcal{F}:M\rightarrow G/K$ is harmonic if and only if
\begin{equation}\label{integr}\mathrm{d}
\alpha_{\lambda}+\frac{1}{2}[\alpha_{\lambda}\wedge\alpha_{\lambda}]=0,\ \ \hbox{for all}\ \lambda \in S^1.
\end{equation}
\end{lemma}

\begin{definition}\label{def-1} Let $\mathcal{F}:M\rightarrow G/K$ be harmonic
and  {let $F: \tilde{M} \rightarrow G$ be a frame satisfying $\mathcal{F}=\pi \circ F$
and $F(z_0,\bar z_0) = e$, stated as above.}
Define $\alpha_{\lambda}$ as   {in \eqref{alphalambda} and consider
on $\tilde{M}$  a solution  $F(z,\bar z, \lambda), \lambda \in \C^*$,  to the equation }
\begin{equation}\label{eq-F-int}
 {\mathrm{d} F(z,\bar z, \lambda)= F(z,\bar z, \lambda)\alpha_{\lambda}, \hspace{2mm} F(z_0,\bar z_0,\lambda) = e.}
\end{equation}

 {As a consequence,  the choice of initial condition stated determines $F(*,*,\lambda)$ uniquely by $F$. Moreover,
we then also obtain  $F(z,\bar z, \lambda)|_{\lambda = 1} = F(z,\bar z)$
for all $z \in \tilde M$.
Any such solution will be called an  {\em extended frame} for the harmonic map $\mathcal{F}$.}

 {Finally we would like to point out that  $\mathcal F_{\lambda}:=F(z,\bar z, \lambda)\mod K$  gives a family of harmonic maps with $\mathcal F_{\lambda}|_{\lambda=1}=\mathcal F$. It will be called the {\bf associated family of $\mathcal F$}.}
 \end{definition}

\begin{remark}
 {In this paper, we will usually assume for a harmonic map the conventions introduced above.
In particular, we will assume (without further saying) the existence of a base point $z_0$ and an extended frame such that  {$F(z_0,\bar z_0, \lambda) = e$} holds. However, in Theorem 4.14 below we will encounter a case, where, a priori,  the initial condition may be necessarily much more involved.  {Fortunately}, in Section 4.5 it will be shown that ``it  suffices to consider the cases where the initial condition in Theorem \ref{thm-finite-uniton2} is the identity matrix".}
\end{remark}

\subsection{Loop groups and decomposition theorems}

 For the construction of harmonic maps we will always employ the loop group method.  In this context we consider the twisted loop groups of $G$ and $G^{\mathbb{C}}$
and some of their frequently occurring subgroups:
\begin{equation*}
\begin{array}{llll}
\Lambda G^{\mathbb{C}}_{\sigma} ~&=\{\gamma:S^1\rightarrow G^{\mathbb{C}}~|~ ,\
\sigma \gamma(\lambda)=\gamma(-\lambda),\lambda\in S^1  \},\\[1mm]
\Lambda G_{\sigma} ~&=\{\gamma\in \Lambda G^{\mathbb{C}}_{\sigma}
|~ \gamma(\lambda)\in G, \hbox{for all}\ \lambda\in S^1 \},\\[1mm]
\Omega G_{\sigma} ~&=\{\gamma\in \Lambda G_{\sigma}|~ \gamma(1)=e \},\\[1mm]
\Lambda^{-} G^{\mathbb{C}}_{\sigma}  ~&=
\{\gamma\in \Lambda G^{\mathbb{C}}_{\sigma}~
|~ \gamma \hbox{ extends holomorphically to }  {D_\infty \}},\\[1mm]
\Lambda_{*}^{-} G^{\mathbb{C}}_{\sigma} ~&=\{\gamma\in \Lambda G^{\mathbb{C}}_{\sigma}~
|~ \gamma \hbox{ extends holomorphically to }D_{\infty},\  \gamma(\infty)=e \},\\[1mm]
\Lambda^{+} G^{\mathbb{C}}_{\sigma} ~&=\{\gamma\in \Lambda G^{\mathbb{C}}_{\sigma}~
|~ \gamma \hbox{ extends holomorphically to }D_{0} \},\\[1mm]
\Lambda_{S}^{+} G^{\mathbb{C}}_{\sigma} ~&=\{\gamma\in
\Lambda^+ G^{\mathbb{C}}_{\sigma}~|~   \gamma(0)\in S \},\\[1mm]
\end{array}\end{equation*}
where $D_0=\{z\in \mathbb{C}| \ |z|<1\}$,  { $D_\infty=\{z\in \mathbb{C}| \ |z|>1\} \cup\{\infty\}$ }and $S$ is some subgroup of $K^\C$.
\vspace{2mm}

If the group $S$ is chosen to be $S = (K^\C)^0$, then we  {write $\Lambda_{\mathcal{C}}^{+} G^{\mathbb{C}}_{\sigma} $.}
If the group $S$ is chosen to be $S = \{e\}$, then we write  {$\Lambda_{\star}^{+} G^{\mathbb{C}}_{\sigma} $.}

  {In some cases subgroups $S$, different from the above, are chosen, like in cases where there exists a Borel subgroup. In such cases one can derive a unique decomposition of any loop group element. In other cases, like in \cite{DoWa12}, one can only derive unique decompositions
for elements in some open subset of the given loop group:
\begin{remark}
If $G = SO^+(1,n+3)$ and $K = SO^+(1,3)\times SO(n)$, then there exists a closed, connected solvable subgroup $S \subseteq (K^\C)^0$ such that
the multiplication $\Lambda G_{\sigma}^0 \times \Lambda^{+}_S G^{\mathbb{C}}_{\sigma}\rightarrow
{(\Lambda G^{\mathbb{C}}_{\sigma})^0}$ is a real analytic diffeomorphism onto the open subset
$ \Lambda G_{\sigma}^0 \cdot \Lambda^{+}_S G^{\mathbb{C}}_{\sigma}     \subset(\Lambda G^{\mathbb{C}}_{\sigma})^0$.
\end{remark}
}

We frequently use the following decomposition theorems
(see \cite{Ke1}, \cite{DPW}, \cite{PS}, \cite{DoWa12}).

\begin{theorem} \label{thm-decomposition}\
\begin{enumerate}
\item {\em (Iwasawa decomposition)}
\begin{enumerate}
\item
$ (\Lambda G^{\C})_{\sigma} ^0=
\bigcup_{\delta \in \Xi }( \Lambda G)_{\sigma}^0\cdot \delta\cdot
\Lambda^{+}_\mathcal{C} G^{\mathbb{C}}_{\sigma},$
where $\Xi $ denotes {a (discrete) }set of representatives for the  double-coset decomposition;
\item The multiplication $\Lambda G_{\sigma}^0 \times \Lambda_\mathcal{C}^{+} G^{\mathbb{C}}_{\sigma}\rightarrow
(\Lambda G^{\mathbb{C}}_{\sigma})^0$ is a real analytic map onto the connected open subset
$ \Lambda G_{\sigma}^0 \cdot \Lambda_\mathcal{C}^{+} G^{\mathbb{C}}_{\sigma}   = {\mathcal{I}_e} \subset \Lambda G^{\mathbb{C}}_{\sigma}$.

 {Here $\mathcal{I}_e$ denotes the (connected) open Iwasawa cell containing the identity element.}
\end{enumerate}
\item  {\em (Birkhoff decomposition)}

\begin{enumerate}
\item
 {$(\Lambda {G}^\C )^0= \bigcup _{\omega \in \mathcal{W}} \Lambda^{-}_{\mathcal{C}} {G}^{\mathbb{C}}_{\sigma} \cdot \omega \cdot \Lambda^{+}_{\mathcal{C}} {G}^{\mathbb{C}}_{\sigma}$
where  $\mathcal{W}$ denotes a (discrete) set of representatives for the double coset
decomposition}

\item The multiplication $\Lambda_{*}^{-} {G}^{\mathbb{C}}_{\sigma}\times
\Lambda^{+}_\FC {G}^{\mathbb{C}}_{\sigma}\rightarrow
\Lambda {G}^{\mathbb{C}}_{\sigma}$ is an analytic  diffeomorphism onto the
open and dense subset $\Lambda_{*}^{-} {G}^{\mathbb{C}}_{\sigma}\cdot
\Lambda^{+}_\FC {G}^{\mathbb{C}}_{\sigma}$ {\em ( big Birkhoff cell )}.
\end{enumerate}
\end{enumerate}
\end{theorem}

\begin{remark} \ \label{S2}
\begin{enumerate}
\item  {We would  like to recall that for inner symmetric spaces  the twisted loop algebras
are isomorphic to the untwisted ones.
For the algebraic case see, e.g. \cite{Kac}, chapter 8, and our case follows by completion in the topology used in  this paper.}
\item  {The middle terms of the Birkhoff decomposition, item $(2)(a)$ above, form, in the untwisted case, the Weyl group
of the corresponding untwisted loop algebra (see e.g. \cite{PS}) and therefore form a discrete subset of the loop group.}
\item  {The middle terms in item $(1)(a)$ above can be determined quite precisely by using \cite{Ke1}.
Roughly speaking, they correspond to factors in a natural product decomposition of the Weyl group elements of the Birkhoff decomposition and thus they form a discrete subset of the loop group as well. For the present paper we will not need any special information about these factors.}
\item  {It is well known that in the Birkhoff decomposition only one of the double cosets is an open subset of the loop group under consideration. Therefore the name "open cell" seems to be appropriate.
In the case of the Iwasawa decomposition, in general, several open double cosets can occur
(see, e.g., \cite{Ke1} for an explicit example and also \cite{D:open cells}).
But the open cell $\mathcal{I}_e$ containing the identity element plays naturally a special role.
Therefore it gets a name.}
\end{enumerate}
\end{remark}

Loops which have a finite Fourier expansion are called {\it algebraic loops} and
 denoted by the subscript $``alg"$, like
$\Lambda_{alg} G_{\sigma},\ \Lambda_{alg} G^{\mathbb{C}}_{\sigma},\
\Omega_{alg} G_{\sigma} $ as in  \cite{BuGu}, \cite{Gu2002}. And we define
  \begin{equation}\label{eq-alg-loop}\Omega^k_{alg} G_{\sigma}:
  =\left\{\gamma\in
\Omega_{alg} G_{\sigma}|
Ad(\gamma)=\sum_{|j|\leq k}\lambda^jT_j \right\}\subset \Omega_{alg} G_{\sigma} .\end{equation}


\subsection{ The DPW method and its potentials}

 With the loop group decompositions as stated above, we obtain a
construction scheme of harmonic maps from a surface into $G/K$.

\begin{theorem}\label{thm-DPW}\cite{DPW}, \cite{DoWa12}, \cite{Wu}.
Let $\D$ be a contractible open subset of $\C$ and $z_0 \in \D$ a base point.
Let $\mathcal{F}: \D \rightarrow G/K$ be a harmonic map with $\mathcal{F}(z_0)=eK.$
The associated family $\mathcal{F}_{\lambda}$ (See Definition \ref{def-1}) of $\mathcal F$  can be lifted to a map
$F:\D \rightarrow \Lambda G_{\sigma}$, the extended frame of $\mathcal{F}$, and we can assume  without loss of generality that $F(z_0,\bar z_0, \lambda)= e$ holds.
Under this assumption,
\begin{enumerate}
\item
 The map $F$ takes only values in
 {$\mathcal{I}_e\subset \Lambda G^{\mathbb{C}}_{\sigma}$, i.e. in the open Iwasawa cell containing the identity element.}

 \item There exists a discrete subset $\D_0\subset \D$ such that on $\D\setminus \D_0$
we have the decomposition
\[F(z,\bar{z},\lambda)=F_-(z,\lambda)  F_+(z,\bar{z},\lambda),\]
where \[F_-(z,\lambda)\in\Lambda_{*}^{-} G^{\mathbb{C}}_{\sigma}
\hspace{2mm} \mbox{and} \hspace{2mm} F_+(z,\bar{z},\lambda)\in \Lambda^{+}_{\FC} G^{\mathbb{C}}_{\sigma}.\]
and $F_-(z,\lambda)$ is meromorphic in $z \in \D$  and satisfies
$F_-(z_0,\lambda) = e$.

Moreover,
\[\eta= F_-(z,\lambda)^{-1} \mathrm{d} F_-(z,\lambda)\]
 {is a $\lambda^{-1}\cdot\mathfrak{p}^{\mathbb{C}} \textendash \hbox{valued}$ meromorphic $(1,0) \textendash$form} with poles at points of $\D_0$ only.

\item Spelling out the converse procedure in detail we obtain:
Let $\eta$ be a  { $\lambda^{-1}\cdot\mathfrak{p}^{\mathbb{C}} \textendash \hbox{valued}$ meromorphic $(1,0) \textendash$form}  such that the solution
to the ODE
\begin{equation}
F_-(z,\lambda)^{-1} \mathrm{d} F_-(z,\lambda)=\eta, \hspace{5mm} F_-(z_0,\lambda)=e,
\end{equation}
is meromorphic on $\D$, with  $\D_0$ as set of possible poles.
 Then on  $\D_{\mathcal{I}}
= \{  z \in \D\setminus {\D_0}\ |\  {F_-(z,\lambda) \in \mathcal{I}_e}\}$
we
define  $\tilde{F}(z,\lambda)$ by the Iwasawa decomposition
\begin {equation}\label{Iwa}
F_-(z,\lambda)=\tilde{F}(z,\bar{z},\lambda)  \tilde{F}_+(z,\bar{z},\lambda)^{-1}.
\end{equation}
 This way one obtains  an extended frame
\[\tilde{F}(z,\bar{z},\lambda)=F_-(z,\lambda)  \tilde{F}_+(z,\bar{z},\lambda)\]
of some harmonic map from $  \D_{\mathcal{I}}  $ to $G/K$  satisfying
$\tilde{F}(z_0,\bar{z}_0,\lambda)= e$.

\item Any harmonic map  $\mathcal{F}: \D\rightarrow G/K$ can be derived from a
 {$\lambda^{-1}\cdot\mathfrak{p}^{\mathbb{C}} \textendash \hbox{valued}$ meromorphic
$(1,0) \textendash$form} $\eta$ on $\D$.
Moreover, the two constructions outlined above  are inverse to each other (on appropriate domains of definition).
\end{enumerate}
\end{theorem}

\begin{remark}\ \label{sphere}
\begin{enumerate}
\item   {A typical application of the theorem above arises as follows: one considers a harmonic map $\mathcal{F}: M \rightarrow G/K$, where $M$ is any Riemann surface and $G/K$ any inner semi-simple symmetric space and considers the natural lift $\tilde{\mathcal{F}}: \tilde{M} \rightarrow G/K$. Then, if $M$ is non-compact or compact of positive genus, then $\tilde{M}$ is contractible and one can apply the theorem above to  $\tilde{\mathcal{F}}$.}

\item  {So the question is: what happens if $M = S^2$? This case has been discussed in detail in Section 3.2 of \cite{DoWa12}. Basically, the theorem above still holds, if one admits  some
singular points. More precisely, if $\mathcal{F}: S^2 \rightarrow G/K$ is harmonic, then  {(see e.g. loc.cit. Theorem 3.11)} after removing at most two  {(different, but otherwise arbitrary)} points $\{p_1,p_2\}$ from $S^2$ one can find an extended frame
$F^\prime : S^2 \setminus{ \{p_1,p_2\}} \rightarrow  \Lambda G_\sigma$ for
$\mathcal{F}^\prime  = \mathcal{F}| S^2 \setminus{ \{p_1,p_2\}}  : S^2 \setminus{ \{p_1,p_2\}}  \rightarrow G/K$. Moreover, $F^\prime_-$, formed by Birkhoff decomposing $F^\prime$, extends  meromorphically  to $S^2$ and the normalized potential formed with $F^\prime_-$ extends meromorphically to $S^2.$
The converse{, the construction of a harmonic map defined on $S^2$ from a normalized potential,} can be carried out  {as usual} if one admits at most two singularities  {in the extended frame associated with the original normalized potential. For more details we refer to loc.cit.}
Of course, if one wants to obtain a harmonic map defined  {and smooth} on all of $S^2$, then additional conditions at the  {poles of the original normalized potential} need to be imposed.
We will use this result at several places below. }
\item
   {The restriction above to factorizations   on  $\D_{\mathcal{I}}$ implies that on this set we have globally an Iwasawa decomposition of the form $(2.6)$ with $F_{\pm}$  globally smooth. This implies, of course, the smoothness of the associated harmonic map. At isolated points of $\D_{\mathcal{I}}$ in $\D$ the frames generally exhibit singular behaviour.
 In some cases, however, the corresponding harmonic maps are nevertheless non-singular.
 When considering, e.g.,  Willmore surfaces, singularities in the frame may or may not induce singularities in the associated harmonic map \cite{DoWa12}, \cite{Wang-3}.}
 {For example, singularities can occur in the extended frame, while both the associated harmonic conformal Gauss map as well as the corresponding Willmore surface stays smooth at the singularities \cite{Wang-3}. See also Example \ref{example} and \cite{Wang-3} for detailed discussions. Therefore, when discussing concrete examples of global immersions one needs to determine separately for all singularities of the frame, whether the final surface has a singularity, i.e. a branch point, or whether it is smooth and an immersion there. We refer to \cite{Wang-3} for examples of Willmore surfaces with branch points.}
\end{enumerate}

\end{remark}

\begin{definition}\cite{DPW},\ \cite{Wu}.
 {With the conventions as above, let  $\mathcal{F}: M \rightarrow G/K$ be a
harmonic map with basepoint $p_0$, $\tilde{\pi}: \tilde{M} \rightarrow M$ the universal cover of $M$
and  $\tilde{\mathcal{F}}: \tilde{M} \rightarrow G/K$
the natural lift of $\mathcal{F}$ with basepoint $z_0,$ where $\tilde{\pi}(z_0) = p_0$.
Then the
$\lambda^{-1}\cdot \mathfrak{p}^{\mathbb{C}} \textendash \hbox{valued}$
 meromorphic $(1,0)\textendash$form  $\eta$  defined in $(2)$ of the last theorem for
 $\tilde{\mathcal{F}}$ is called the {\em normalized potential} for the harmonic
map $\mathcal{F}$ with the point $z_0$ as the reference point. And $F_-(z,\lambda)$ given above is called the  {corresponding} meromorphic extended frame.}
\end{definition}

The normalized potential is uniquely determined,  {since the extended frames are normalized to $e$ at some fixed base point on $\tilde{M}$.}
The normalized potential is  meromorphic  {in $z \in \tilde{M}$.}

In many applications it is much more convenient to use potentials which have a Fourier expansion containing more than one power of $\lambda$.
And when permitting many (maybe infinitely many) powers of $\lambda$,  one can
 {obtain holomorphic coefficients:}

\begin{theorem}\cite{DPW}, \cite{DoWa12}.\label{thm-CC}
Let $\D$ be a contractible open subset of $\C$.
Let $F(z,\bar{z},\lambda)$ be the frame of some harmonic map
into $G/K$. Then there exists  {some real-analytic $V_+: \D  \rightarrow  \Lambda^{+} G^{\mathbb{C}}_{\sigma} $ such that $C(z,\lambda) =
F(z, \bar z, \lambda)  V_+ (z, \bar z,\lambda) $} is holomorphic in $z\in\mathbb{D}$ and in $\lambda \in \mathbb{C}^*$.
Then the Maurer-Cartan form $\eta = C^{-1} \mathrm{d} C$ of $C$ is a holomorphic
$(1,0) \textendash$form on $\D$ and it is easy to verify that $\lambda \eta$ is holomorphic for $\lambda \in \C$.

Conversely, Let $\eta\in\Lambda\mathfrak{g}^{\C}_{\sigma}$ be a holomorphic
$(1,0)\textendash$form such that $\lambda \eta$ is holomorphic  in $\lambda$ for $\lambda \in \C$, then by the same process  {as} given in Theorem \ref{thm-DPW} we obtain a harmonic map $\mathcal{F}: \D \rightarrow G/K$.
\end{theorem}

 {\begin{definition}\label{rm-C}
 The matrix function $C(z,\lambda)$ associated with the holomorphic $(1,0) \textendash$form $\eta$ as in Theorem \ref{thm-CC} will be called a {\em holomorphic extended frame} for the harmonic map $\mathcal{F}$.
 \end{definition}}

\subsection{Symmetries and monodromy}

 {It is natural to investigate harmonic maps with symmetries.}
Since harmonic maps frequently occur as ``Gauss maps" of some surfaces,
 {the
investigation  of harmonic maps with symmetries also has  implications for surface theory.
}

\begin{definition}
 {Let  $\mathcal{F}: M \rightarrow G/K$ be a harmonic map. Then a pair, $(\gamma,R)$, is called a symmetry of $\mathcal{F}$, if
$\gamma$ is an automorphism of $M$ and R is an automorphism of $G/K$ satisfying
 $$\mathcal{F}(\gamma.p) = R.\mathcal{F}(p)$$
  for all $p \in M$.}
  \end{definition}

   {We would like to point out that an intuitive notion of "symmetry" for $\mathcal{F}$ would be an automorphism $R$ of $G/K$ such that  $ R \mathcal{F}(M) = \mathcal{F}(M)$. In some cases one can prove that this intuitive definition implies the actual definition given just above.}

  \begin{lemma} \label{frametransform}
 {  Let  $\mathcal{F}: M \rightarrow G/K$ be a harmonic map {and}  $(\gamma,R)$ a symmetry
  of $\mathcal{F}.$
Let $\tilde{M}$ denote the universal cover of $M$
and $\tilde{\mathcal{F}}: \tilde{M} \rightarrow G/K, \tilde{\mathcal{F}} = \mathcal{F} \circ \pi$, its natural lift. Then:
 \begin{enumerate}
 \item  $\tilde{\mathcal{F}}$ satisfies
\[\tilde{\mathcal{F}} (\gamma.z) = R.\tilde{\mathcal{F}}(z).\]
\item
 For any frame $F : \tilde{M} \rightarrow G$  of $\mathcal{F}$ one  obtains
\begin{equation} \label{symmetry-nolambda}
 {\gamma^*F(z,\bar z) = RF(z,\bar z)k(z,\bar z),}
\end{equation}
where $k(z,\bar z)$ is a function from $\tilde{M}$ into $K$.
\item  For the extended frame  {$F(z,\bar z,\lambda) : \tilde{M} \rightarrow \Lambda G_\sigma$
of $\mathcal{F}$} there exists some map
 $\rho_\gamma: \C^* \rightarrow \Lambda G_\sigma$ such that
 \begin{equation} \label{symmetry}
\gamma^*F(z,\bar z,\lambda)  = \rho_\gamma (\lambda) F(z,\bar z,\lambda)  {k(z, \bar z)},
\end{equation}
where $k$ is the $\lambda \textendash$independent function from $\tilde{M}$ into $K$ occurring in the previous equation.
Moreover, $\rho_{\gamma}(\lambda)|_{ \lambda= 1} = R$ holds.
  \end{enumerate}}
\end{lemma}

 {Note, since $\mathcal{F}$ is full, for each symmetry $(\gamma,R)$ the automorphism $R$ of $G/K$
is uniquely determined by $\gamma$. We therefore write $\rho_\gamma (\lambda) = \rho(\gamma,\lambda)$ and ignore $R$ in this notation. Also note that $\rho_\gamma$ actually is defined and holomorphic for all $\lambda \in \C^*$.}

\begin{proof}
 {The first two equations follow immediately from the definitions. In view of (\ref{alphalambda})
the equality of Maurer-Cartan forms of (\ref{symmetry-nolambda}) implies the equality of the Maurer-Cartan forms of (\ref{symmetry}). Therefore only the last statement needs to be proven. But evaluating  (\ref{symmetry}) at the base point $z_0$ of $\tilde{\mathcal{F}}$ yields
 {$F(\gamma.z_0,\overline{\gamma.z_0},\lambda)  = \rho_\gamma (\lambda) k (z_0,\bar z_0).$} Hence we obtain
$\rho_\gamma: \C^* \rightarrow \Lambda G_\sigma,$ and putting $\lambda = 1$  we infer
 {$F(\gamma.z_0,\overline{\gamma.z_0}) = \rho_\gamma k(z_0,\bar z_0).$} On the other hand, from (\ref{symmetry-nolambda})
we obtain  {$F(\gamma.z_0,\overline{\gamma.z_0}) = R k(z_0,\bar z_0)$}, whence $R =\rho_{\gamma}(\lambda)|_{ \lambda= 1} $.}
\end{proof}

\begin{definition}
 {With the notation above,  the matrix $\rho_\gamma (\lambda), \lambda \in S^1,$ ( for all $\lambda\in\C^* $ in fact) occurring in (\ref{symmetry})  is called the
monodromy (loop)  matrix of $\gamma$  for $\mathcal{F}.$}
\end{definition}

The following result has been proven in Theorem 4.8 of  {\cite{Do-Wa-sym}.}

\begin{theorem}
Let $M$ be a Riemann surface which is either non-compact or compact of positive genus.
\begin{enumerate}
\item
 Let $\mathcal{F}:M \rightarrow G/K$ be a harmonic map and
  {$\tilde{\mathcal{F}}: \tilde{M} \rightarrow G/K$  its natural lift to the universal cover $\tilde{M}$ of $M$.} Then there exists a normalized potential and a holomorphic potential for $\mathcal{F}$, namely the corresponding  {potential} for $\tilde{\mathcal{F}}$.

\item Conversely, starting from some potential producing a harmonic map
$\tilde{\mathcal{F}}$ from $\tilde{M}$ to $G/K$, one obtains a harmonic map $\mathcal{F}$ on $M$ if and only if

\begin{enumerate}
\item  The monodromy matrices $\chi(g, \lambda)$ associated with
$g \in \pi_1 (M)$, considered as automorphisms of $\tilde{M}$, are elements of $(\Lambda G_{\sigma})^0$.

\item  There exists some $\lambda_0 \in S^1$ such that  {$\chi(g, \lambda)|_ { \lambda= \lambda_0} =e$, i.e.
\begin{equation*}
\begin{split} F(g.z,\overline{g.z},\lambda)|_{ \lambda= \lambda_0}&\equiv
\chi(g, \lambda)|_{ \lambda= \lambda_0} F(z, \bar{z}, \lambda)|_{ \lambda= \lambda_0}\ \mod\ K\\
&\equiv   F(z, \bar{z}, \lambda)|_{ \lambda= \lambda_0}\ \mod\ K
\end{split}
\end{equation*}}
for all $g \in \pi_1 (M)$.

\end{enumerate}

\end{enumerate}
\end{theorem}

 {We also need the existence of normalized potentials for harmonic maps from a $2-$sphere
to an inner symmetric space, compact or non-compact.  An important difference to the previously discussed cases is that the extended frames will not be smooth globally on $S^2$,  but will be smooth (actually real analytic) on
$S^2 \setminus {\hbox{\{two points\}}}.$  See Remark \ref{sphere}.
The details can be found in $(2)$ of Remark \ref{sphere}
or Theorem 3.11 of \cite{DoWa12} which  we recall for the convenience of the reader:}

\begin{theorem}\label{normalized-potential-sphere} (Theorem 3.11 of  \cite{DoWa12} )
 {Every  harmonic map from $S^2$ to any  Riemannian or pseudo-Riemannian
  symmetric space  $G/K$ admits an extended frame with at most two singularities and it admits a global meromorphic extended frame. In particular, every  {spacelike conformal }  harmonic map from $S^2$ to any  Riemannian or pseudo-Riemannian  symmetric space  $G/K$ can be obtained from some meromorphic normalized potential.}
\end{theorem}

 {While we generally restrict in this paper to inner symmetric spaces, the last and the following theorem are stated more generally because of their importance.  As stated before, the case of outer symmetric spaces will not be considered in any detail in this paper.}

 { If $M$ has a non-trivial fundamental group, then the invariant potentials are  of particular interest. The first part of the theorem below (the case of a non-compact $M$) has been proven first  in \cite{Do-Ha2} for the case of  harmonic maps into $S^2$.   The case of harmonic maps from a compact Riemann surface  $M$
 into the (inner) symmetric space $G/K=SO^+(1,n+3)/SO^+(1,3)\times SO(n)$ was treated in
 Section 6 of  \cite{Do-Wa-sym}. }

  {For the case of compact $M$  no proof for a general inner symmetric space  is known.
 For the case of non-compact $M$  the published proofs are not very clear nor explicit.
 We therefore include a new proof using work of Bungart \cite{Bungart} and R\"ohrl \cite{Roehrl}.}

\begin{theorem} \label{thm-monodr}
Let $M$ be a Riemann surface.

\begin{enumerate}
\item If $M$ is non-compact, then every harmonic map from $M$ to any symmetric space  {$G/K$} can be generated from some holomorphic potential on $M$, i.e. it can be generated from some holomorphic potential on the universal cover $\tilde{M}$ of $M$ which is invariant under the fundamental group of $M$.

 \item If $M = S^2$, then every harmonic map from $S^2$ to any symmetric space can be generated from some meromorphic  potential on $S^2$.

\item If $M$ is compact, but different from $S^2$, then every harmonic map from $M$ to
 {the inner symmetric space  $G/K = SO^+(1,n+3)/SO(1,3)\times SO(n)$}  can be generated by some meromorphic  potential defined on $M$, i.e. from  a  meromorphic potential defined on the (contractible) universal cover $\tilde{M}$ of $M$ which is invariant under the fundamental group of $M$.
\end{enumerate}
\end{theorem}

\begin{proof}
  {For item $(2)$ see  $(2)$  of Remark \ref{sphere} and for item $(3)$ see Section 6 of \cite{Do-Wa-sym}. So let us assume now that $M$ is non-compact, $G/K$ any symmetric space
and $\mathcal{F} : M \rightarrow G/K$ a harmonic map. Let $F(z,\bar z,\lambda)$ denote an extended frame
for $\mathcal{F}.$
Then  from Lemma 4.11 in \cite{DPW} we obtain a  $($real analytic$)$ matrix function
 $\tilde V_+: \tilde{ M} \rightarrow \Lambda^+ G^\C_\sigma$ such that the matrix
 $C$ defined by
\begin{equation}
 C(z,\lambda) := F(z,\bar z,\lambda) \tilde{V}_+(z,\bar z, \lambda)
\end{equation}
is holomorphic in $z \in \tilde{M}$ and $\lambda \in \C^*$.
 Moreover, $F(z,\bar z,\lambda)$ satisfies  (\ref{symmetry}).
 As a consequence,  $C$ inherits
 from $F(z,\bar z,\lambda)$ the transformation behaviour
  \begin{equation}
  C(\tau.z, \lambda) = M(\tau, \lambda)  C(z,\lambda) W_+(\tau, z, \lambda),
 \end{equation}
where $\tau \in  \pi_1(M)$  and  $W_+: \tilde{ M} \rightarrow  \Lambda^+ {G^\C}_\sigma$
 is holomorphic in $z$ and $\lambda \in \C^*$.  It is straightforward to verify the ``cocycle condition"
 \begin{equation}
W_+(\tau \mu,z,\lambda) =
W_+(\tau, \mu.z,\lambda) W_+(\mu,z,\lambda)  \hbox{ for all  $\tau, \mu \in \pi (M ).$}
\end{equation}
Our goal is to split the cocycle $W_+(\tau,z,\lambda)$ in $\Lambda^+ G^\C_\sigma$. For this we begin by following  the first few lines of the proof of Theorem 3 of  \cite{Roehrl}.The paper refers to complex Lie groups,
  which, in our case we consider the complex Banach Lie group $H = \Lambda^+ G^\C_\sigma$.}
   If $H$ is a complex Banach Lie group, then we denote by
 ${(H_\omega})^\mathcal{C}$  the sheaf of  continuous sections
from open subsets of $M$ to $H$.
Similarly, by  ${(H_\omega})^\mathcal{H}$ we denote the sheaf of  holomorphic sections
from open subsets of $M$ to $H$.

  {First we prove:
\begin{itemize}
\item[(a)]  Let $M$ be a non-compact Riemann surface and $H^\C$ a complex Banach Lie group
then $H^1(M, {(H_\omega)}^\mathcal{C})$ = 0.
\end{itemize}
}

For the convenience of the reader we translate the first 10 or so
lines of this proof:
For $\xi\in H^1(M,H)$  we need to prove that in the principal bundle associated
with $\xi$ there exists a continuous section. Since the principal bundle can contain, for dimension reasons, at most two-dimensional obstructions, it suffices for the existence of a continuous section the verification that the two-dimensional obstruction vanishes. But this obstruction is an element of $H^2(M, \pi_1(H))$. Moreover, for a non-compact Riemann surface $M$ it is known that $H_2(M,\mathbb{Z})$ vanishes, whence by the universal coefficient theorem also $H^2(M, \pi_1(H))$ vanishes. This proves claim \textrm{(a)}.
 From this we derive the complex Banach group version of
 \cite[Theorem 3]{Roehrl}:
\begin{itemize}
\item[(b)]
 Let $M$ be a non-compact Riemann surface and $H$ a complex Banach Lie group
then $H^1(M, {(H_\omega)}^\mathcal{H}) = 0.$
\end{itemize}
 {But \cite[Theorem 8.1]{Bungart} implies}
\[
 H^1(X, {(G^\C_\omega)}^\mathcal{C}) \cong H^1(X, {(G^\C_\omega)}^\mathcal{H}),
\]
and claim $(b)$ follows.

 To finish the proof of the splitting  theorem  for the cocycle  $W_+(\tau, z, \lambda),$ we can now use
 \cite[Exercise 31.1]{Forster}.
 For a detailed proof one can follow the proof of
 \cite[Theorem 31.2]{Forster}, but with
 $\Psi_i(z)=W_+(\eta_i(z)^{-1},z,
 \lambda)^{-1}$ where we refer for notation to loc.cit.

 We now know $ W_+(\tau, z, \lambda) = P_+(z,\lambda)  P_+(\tau.z,\lambda)^{-1} $ and consider
$\hat{C} =  C P_+$.
A simple computation shows
$ \hat{C}(\tau.z,\lambda) = M(\tau, \lambda) C(z, \lambda)$
  for all $\tau \in \pi_1(M)$ and all
 $\lambda \in \C^*$.

   {As a consequence, the differential one-form $ \eta = \hat{C}^{-1} \dd\hat{C}$
 is invariant under $\pi_1(M)$.
 This finishes the proof of the theorem.}

  \end{proof}

 { The differential $1-$form $ \eta = \hat{C}^{-1} \dd\hat{C}$ as just above we will call
 an {\rm invariant} holomorphic potential for the   harmonic map $\mathcal{F}$.}


\section{Harmonic maps of finite uniton type}

This section aims at interpreting all harmonic maps of finite uniton type in terms of loop group language. For this purpose, we focus first on simply-connected Riemann surfaces. Then we apply the results to Riemann surfaces with {non-trivial} topology, where monodromy appears naturally. We also provide the relations between harmonic maps of finite uniton type into non-compact symmetric spaces and their dual harmonic maps into compact symmetric spaces, which is vital for later applications. We end this section by some discussions on dressing actions on harmonic maps of finite uniton type.

\subsection{Finite uniton type harmonic maps defined on simply-connected Riemann surfaces} \label{f.u.1-connM}
In this subsection we denote by $\tilde{M}$ a simply-connected Riemann surface, i.e., $S^2$,  $\E=\{z\in\C\ | \ |z|<1 \}$, or $\C$. Let $G/K$ denote an inner symmetric space and $\mathcal{F}:\tilde{M}\rightarrow G/K$ a harmonic map. We would like to recall our conventions:  {Let $\mathcal{F}:\tilde{M} \rightarrow G/K$ be a harmonic map defined on a simply-connected Riemann surface $\tilde{M}$. Then we assume  {at least tacitly} that a basepoint  $z_0\in \tilde {M}$  is chosen and we also assume that any extended frame $F$, the corresponding normalized  extended frame $F_-$  and  {any} holomorphic extended frame $C$  {associated with an extended frame $F$} all attain the  $e$ at $z_0$.}
 {The case of $M = S^2$ has been addressed in item  {$(2)$} of Remark \ref{sphere}, including the reference given there.}

\begin{definition} (Finite uniton  {type -- simply-connected}  Riemann surface ) Let $\tilde{M}$ be a simply-connected Riemann surface. A harmonic map $\mathcal{F}:\tilde{M}\rightarrow G/K$  {is said} to be of finite uniton type if some  extended frame $F$ of $\mathcal{F},$ satisfying $F(z_0,\bar z_0, \lambda)=e$ for some base point $z_0\in \tilde{M},$ is a Laurent polynomial in $\lambda$.
\end{definition}

 {Note that the condition
 {$F(z_0,\bar z_0,\lambda) = e$ is not a restriction. Assume some $\hat{F}$
satisfies all conditions  of the definition above, except $\hat{F}(z_0,\bar z_0,\lambda) = e$,
then $F(z,\bar z,\lambda) = \hat{F}(z_0,\bar z_0,\lambda)^{-1} \hat{F}(z,\bar z,\lambda)$} satisfies all conditions.
We would also like to point out that for any extended frame $F$ of some harmonic map
which is a Laurent polynomial in $\lambda$, both factors in the unique meromorphic Birkhoff decomposition  {$F(z,\bar z,\lambda) = F_-(z,\lambda) F_+(z,\bar z,\lambda)$}, i.e. assuming
$F_-(z,\lambda) = e + \mathcal{O}(\lambda^{-1})$, also are Laurent polynomial.}

\begin{proposition}\label{prop-fut}   Let $\mathcal{F}:\tilde{M} \rightarrow G/K$ be a harmonic map defined on a  {contractible} Riemann surface $\tilde{M}$.
 { Let $z_0\in \tilde {M}$ be a base point. Then the following statements are equivalent:}
\begin{enumerate}
\item   $\mathcal{F}$ is  of finite uniton type.
\item There exists an extended frame  {$F(z,\bar z,\lambda)$} of $\mathcal{F}$ which is a Laurent polynomial in $\lambda$.
\item  The normalized extended frame  {$\hat{F}_-(z,\lambda)$ of any extended frame
$\hat{F}(z,\bar z,\lambda)$ of $\mathcal{F}$} is a Laurent polynomial in $\lambda$.
\item  Every holomorphic  {extended frame $\hat{C}(z,\lambda)$}  associated with any extended frame
 {$\hat{F}(z,\bar z,\lambda)$} of $\mathcal{F}$ only  contains finitely many negative powers of $\lambda$.
\item There exists a holomorphic extended  { frame $C^\sharp(z,\lambda)$ } which  { only contains} finitely many negative powers of $\lambda$.
\end{enumerate}
    For the case $\tilde{M} = S^2,$  the above   {equivalences remain} true, if one replaces in the last two statements the word ``holomorphic" by ``meromorphic" and  {also admits for all
    frames  at most two singularities.}
\end{proposition}

\begin{proof}
 {For a contractible domain, the definition of "finite uniton type" for a harmonic map  can be
rephrased by $(1) \Leftrightarrow (2)$. Hence we only need} to show that  $(2), (3), (4)$ and $(5)$ are equivalent.

(2) $\Rightarrow $  (3):  {Let $F(z,\bar z,\lambda)$ and $\hat{F}(z,\bar z,\lambda)$ be as in $(2)$ and $(3)$ respectively. Then
$\hat{F} (z,\bar z,\lambda)= F (z,\bar z,\lambda)k(z,\bar z,\lambda)$. Inserting the unique Birkhoff decompositions $F(z,\bar z,\lambda) = F_-(z,\lambda)F_+(z,\bar z,\lambda)$ and $\hat{F}(z,\bar z,\lambda) = \hat{F}_- (z,\lambda)\hat{F}_+(z,\bar z,\lambda)$  we infer $F_-(z,\lambda) = \hat{F}_-(z,\lambda)$ and the claim follows.}

(3) $\Rightarrow$ (4):  {By Theorem \ref{thm-CC},}  { $ \hat{C}(z,\lambda) = \hat{F} (z,\bar z,\lambda)\hat{V}_+(z,\bar z,\lambda)$, where
$\hat{V}_+ $ is actually real-analytic.
 Inserting the   Birkhoff decomposition $\hat{F}(z,\bar z,\lambda) = \hat{F}_-(z,\lambda) \hat{F}_+(z,\bar z,\lambda)$, we obtain
 $\hat{C} (z,\lambda)= \hat{F}_-(z,\lambda) \hat{F}_+(z,\bar z,\lambda) V_+(z,\bar z,\lambda) $ and the claim follows.}

(4) $\Rightarrow$ (2):  {Let $\hat{C}(z,\lambda)$ be any holomorphic extended frame for $\mathcal{F}$. Consider the Iwasawa decomposition  $F(z,\bar z,\lambda)=\hat{C}(z,\lambda)  \tilde{C}_+(z,\bar z,\lambda)$ near $z_0$, where  $\tilde{C}_+\in \Lambda^{+} G^{\mathbb{C}}_{\sigma}$
and where $F(z,\bar z,\lambda)$ and  $\tilde{C}_+(z,\bar z,\lambda)$ attain the  $e$ at $z_0$.
Since $\hat{C}(z,\lambda) $ only contains finitely many negative powers of $\lambda$ by assumption, also $F(z,\bar z,\lambda)$ contains only finitely many negative powers of $\lambda$. Since $F(z,\bar z,\lambda)$ is real, it is a Laurent polynomial.}

(4) $ \Leftrightarrow$ (5): Note that for any two holomorphic extended frames $C_1(z,\lambda)$ and $C_2(z,\lambda)$ there exists some $W_+(z,\lambda)\in \Lambda^{+}G^{\mathbb{C}}_{\sigma}$ such that $C_1(z,\lambda)=C_2(z,\lambda)W_+(z,\lambda)$ holds.

 For the case $\tilde{M}=S^2$,   {using the results just proven} for $M_1=S^2\setminus\{\infty\}$ and $M_2=S^2\setminus\{0\}$ respectively, one will obtain  {the last claim of the proposition. For more details see
 $(2)$  of Remark \ref{sphere} and Section 3.2 of \cite{DoWa12}.}
 \end{proof}

We wonder, under what conditions the
based (at $z_0$) normalized  extended frame  {$F_-(z,\lambda)$} will be a Laurent polynomial.
 Let $\eta$ denote the normalized potential  { $F_-(z,\lambda)^{-1} \mathrm{d} F_-(z,\lambda) = \eta$.
 Then $F_-(z,\lambda), F_-(z,\lambda)|_{z=z_0} = e$}, can be obtained from $\eta$ by an application  of the  {standard Picard iteration of the theory of ordinary differential equations. Since $\eta$ is a multiple of
 $\lambda^{-1},$} it is easy to see that each step of the Picard iteration
  {decreases the occurring power of $\lambda$ by $-1$.} So $F_-(z,\lambda)$ is a Laurent polynomial if and only if the Picard iteration stops after finitely many steps. { See for example, Section 1 of \cite{Gu2002}}.
The most natural reason for the  Picard iteration to stop is that the  normalized potential $\eta$ takes values in some nilpotent Lie algebra (Note: If $\eta(z)$ is only nilpotent for every $z\in \tilde M$, it does not follow, in general, that the Picard iteration will stop.).

The following result is a slight generalization of Appendix B of \cite{BuGu}, (1.1) of \cite{Gu2002}. We would like to point out  that in particular $G$ does not need to be compact.

\begin{proposition}  Let $\mathfrak{n}\subset\mathfrak{g}^{\C}$ be a nilpotent subalgebra and assume that $\eta$ is a (  holomorphic or meromorphic ) potential of some harmonic map such that  $\eta(z)\in\mathfrak{n}$ for all
 $z\in \tilde M \setminus \lbrace$ poles of $\eta \rbrace $, and $\eta$ only contains finitely many positive powers of $\lambda$.  Let $\mathcal{F}:\tilde{M}_0\rightarrow G/K$ be the harmonic map associated with $\eta$ on an open subset $\tilde{M}_0\subset\tilde{M}$. Then  $\mathcal{F}$ is of finite uniton type.
\end{proposition}

\begin{proof}
When  $\eta(z)\in\mathfrak{n}$ for all $z\in \tilde M$, then it is easy  to see that the Picard iteration producing the solution
 {$\mathrm{d} C(z,\lambda)= C(z,\lambda) \eta$,} $C(z_0,\lambda)=e$, stops after finitely many steps.
Since $\eta$ is a Laurent polynomial in $\lambda$,  {so is the solution} $C$. By Proposition \ref{prop-fut}, $\mathcal{F}$ is of finite uniton type.
\end{proof}

\begin{remark}
\
\begin{enumerate}
\item
The conformal Gauss maps of many Willmore surfaces (see \cite{DoWa11}, \cite{Wang-1} and \cite{Wang-3}) can be constructed as discussed in the proposition above.
In these examples  the group $G$ is non-compact.

 \item For harmonic maps into arbitrary Lie  groups $G$ with a bi--invariant non--degenerate metric, one can also produce harmonic maps of finite uniton type following the above procedure by considering $G$ as the symmetric space
$(G \times G)/G$. Harmonic $2-$spheres in $U(4)$
provide standard examples for the proposition (see Section 5 of \cite{BuGu} or Appendix B of \cite{Gu2002}). In this case the group $G = U(4)$ is compact.

\end{enumerate}
\end{remark}

  For the construction of examples of Willmore spheres it is very important to know that harmonic maps defined on $M=S^2$ are of finite uniton type.
 First of all  { we reiterate that the loop group approach in the sense of this paper,  {i.e. the DPW method,} using extended frames and not extended solutions,
also applies to harmonic maps from $S^2$ to any compact or non-compact
inner symmetric space ( see Remark \ref{sphere} above). }

 {Using the same approach one can retrieve the result of Uhlenbeck Theorem \cite{Uh},
Segal \cite{Segal}, Burstall-Guest \cite{BuGu}  for compact $G/K$
and generalize it to the case of non-compact inner symmetric spaces  $G/K$.}

\begin{theorem} \label{S2isfu}(Theorem 3.6, of \cite{DoWa13})  For every compact or non-compact {inner} symmetric space $G/K$ every harmonic map $\mathcal{F}: S^2 \rightarrow G/K$ is
of finite uniton type.
\end{theorem}


\subsection{ {Harmonic maps from Riemann surfaces to inner symmetric spaces which have a trivial monodromy representation}} \label{trivmonorep}

In the last subsection we considered exclusively simply-connected Riemann surfaces $M$.
Obviously then, the monodromy representation of any harmonic map $\mathcal{F}:M \rightarrow G/K$
is trivial. But there also exist non-simply-connected Riemann surfaces which admit harmonic maps
$\mathcal{F}:M \rightarrow G/K$ of  {``finite uniton type"} and thus have, by definition,  a trivial monodromy representation.   {Next we introduce the notion of ``finite uniton type".}

\subsubsection{ {Definition of ``finite uniton type"}}

In \cite{BuGu}, Burstall and Guest generalize the work of Uhlenbeck on harmonic maps
$\mathcal{F}:S^2 \rightarrow G$, $G=U(n)$,  to harmonic maps of ``finite uniton number" from a general Riemann surface into compact Lie groups  $G$  as well as compact inner symmetric spaces $G/K$.

 In our  {definition  below}  we separate out two  properties/parts contained implicitly in the definition of \cite{BuGu}, just above  loc.cit.,  Theorem 1.2,  and we also admit non-compact Lie groups and inner symmetric spaces.   The first property is that one considers maps which actually are defined on $M$, whence do not have any monodromy. The second property is that one can assume w.l.g. that the extended solution used in \cite{BuGu} is a Laurent polynomial in $\lambda$. These two properties define two classes of harmonic maps, the intersection of which gives the harmonic maps of finite uniton number. It would be interesting to investigate these two classes separately.

\begin{definition}\label{def-uni} (Finite uniton type -- arbitrary Riemann surface) Let $M$ be a Riemann surface  {and let $G/K$  {be} an inner symmetric space}. A harmonic map
$\mathcal{F}:M\rightarrow G/K$ is said to be of finite uniton type if there exists some extended frame $F$ of $\mathcal{F}$,  defined on the universal cover
$\tilde{M}$ of  {$M,$ satisfying $F(z_0,\bar{z}_0,\lambda)=e$}
and having the following two properties:
\begin{enumerate}[$(U1)$]
\item
 For all $\lambda \in S^1,$ the extended frame $F(z,\bar{z},\lambda): \tilde{M} \rightarrow (\Lambda  {G _{\sigma})} $ has trivial monodromy and {, in view of the last equation of
  Lemma \ref{frametransform}, therefore descends to a well defined map on $M$, i.e.
    $F(z,\bar{z},\lambda): M \rightarrow (\Lambda G_{\sigma})/K$,}
up to two singularities in the case of $M = S^2$.
\item $F(z,\bar{z},\lambda)$ is a Laurent polynomial in $\lambda$.
\end{enumerate}
We will say ``$\mathcal{F}$ has a trivial monodromy representation" if (U1) is satisfied.

 {We would also like to point out that we will use, by abuse of notation,  the same notation for the frame $F$ with values in
$\Lambda G _{\sigma}$ and its projection with values in $ \Lambda G_{\sigma}/K$.}

 \end{definition}

 \begin{remark}\
  \begin{enumerate}
  \item Note that  in Section 4 below we discuss the relations between the DPW theory (using extended frames) and the Burstall-Guest theory (using extended solutions).  It turns out that  {the  Definition \ref{def-uni} for
  the notion of  a finite uniton type harmonic map is equivalent to the  notion of a minimal  (finite) uniton number
 harmonic  map defined on page 546 of \cite{BuGu}. ( Also see Proposition \ref{typeequivnumber}.)}

 \item
We mainly want to use objects occurring in our approach (like extended frames) and also want to make it as clear as possible to what extent the two properties which make up the notion of
 {finite uniton type and minimal (finite) uniton number respectively  contribute to statements and proofs.}

\item Condition $(U1)$ is a very strong condition. Of course, in the case of a non-compact simply connected Riemann surface $M$   {and  {of} $S^2$ respectively,} it is always satisfied.  {We would like to point out that although the assumption $(U1)$ was not stated in \cite{BuGu} explicitly, it was used implicitly by using exclusively extended solutions defined on all of $M$, like in \cite{BuGu}, Theorem 4.5, where $M$ is an
 {arbitrary, compact or non-compact,} Riemann surface.}
\end{enumerate}
\end{remark}

\subsubsection{ {Harmonic maps with trivial monodromy}  {representation}}

 {In this subsubsection we consider all harmonic maps satisfying property $(U1)$. Recall, that we assume that all harmonic maps in this paper are considered to be full,
 {see Definition \ref{deffull}.} In the following proposition we will use, as in the definition just above, by abuse of notation the same notation for the maps $C $ and $F_-$
and their projections respectively to some image quotient space.}

\begin{proposition} \label{prop-frame}
  Let   {$M \neq S^2$}   be any Riemann surface and let $\mathcal{F}: M\rightarrow G/K$ be a harmonic map which is full in $G/K$.  Then the following statements are equivalent

   \begin{enumerate}
\item The  monodromy matrices $\chi(g,\lambda)$, $g\in\pi_1(M)$, $\lambda\in S^1$, satisfy $\chi(g,\lambda)=e,\ \hbox{ for all } g\in\pi_1(M),\ \lambda\in S^1$.

\item Any extended frame $F$ of $\mathcal{F}$ is defined
on $M$ ``modulo $K$'', i.e. $F$ descends from a map defined on $\tilde{M}$ with values in
${\Lambda G}_{\sigma}$ to
$F(z, \bar{z},\lambda ): M \rightarrow  ({\Lambda G}_{\sigma}) / K$.

\item Any holomorphic extended frame $C$ of $\mathcal{F}$ is defined on $M$
``modulo $\Lambda^+G^{\mathbb{C}}_{\sigma}$'' , i.e. $C$ descends from a map defined on $\tilde{M}$ with values in   $\Lambda G^{\C}_{\sigma}$ to
$C(z,\lambda):M\rightarrow \Lambda G^{\C}_{\sigma}/\Lambda^+ G^{\C}_{\sigma}$.

 \item The normalized   extended frame $F_-$ relative to any base point is defined on $M$, i.e. $F_-$ descends from a map defined on $\tilde{M}$ with values in $\Lambda^- G^{\C}_{\sigma}$ to  $F_-(z,\lambda):M\rightarrow \Lambda^- G^{\C}_{\sigma}$.
 \end{enumerate}
  {For $M = S^2$ the four statements just above still hold in view of  {(2) of Remark \ref{sphere} of Section 2,} if one admits up to two singularities {in the extended frames.}}
 \end{proposition}

\begin{proof}
 {In view of Theorem 3.6  the case $M = S^2$ is trivially satisfied. Hence let us assume $M$ is any Riemann surface different from $S^2$.}
By general theory we have for every $g\in\pi_1(M)$ and all $\lambda\in S^1$
on the universal cover $\tilde{M}$ the equations:
\[\begin{split}
F(g.z,\overline{g.z},\lambda)&= \chi(g,\lambda) F(z,\bar{z},\lambda) K(g,z,\bar{z}),\\
C(g.z,\lambda)&=\chi(g,\lambda) C(z,\lambda) W_+(g,z,\bar{z},\lambda),\\
F_-(g.z,\lambda)&=\chi(g,\lambda) F_-(z,\lambda) L_+(g,z,\bar{z},\lambda),\\
  \end{split}\]
for some maps
$K(g,z,\bar{z}): \tilde{M} \rightarrow K $,
$ W_+ (g,z,\bar{z},\lambda),
\ L_+(g,z,\bar{z},\lambda): \tilde{M}\rightarrow \Lambda^+ G^{\C}_{\sigma}$, and $\chi(g,\lambda)\in\Lambda G_{\sigma}$. Therefore (1) implies (2), (3) and (4).

``(2) $\Rightarrow$ (1)": From the assumption we obtain $F(g.z,\overline{g.z},\lambda)= F(z,\bar{z},\lambda)\tilde{K}(g,z,\bar{z})$. This shows \[\chi(g,\lambda)F(z,\bar{z},\lambda)K(g,z,\bar{z})=F(z,\bar{z},\lambda)\tilde{K}(g, z,\bar{z}).\] Since $\mathcal{F}\equiv F\mod K$, this equation implies
\[\mathcal{F}(z,\bar{z},\lambda)=\chi(g,\lambda)\mathcal{F}(z,\bar{z},\lambda), \hbox{ for all } z\in M,\ \lambda\in S^1.\]
When $\mathcal{F}$ is full, $\chi(g,\lambda)=e$ follows.

``(4) $\Rightarrow$ (2)": If $F_-$ is defined on $M$, then  $F_-(g.z,\lambda)=F_-(z,\lambda)$ for all $g\in \pi_1(M)$, $\lambda\in S^1.$ But then $F=F_-F_+$ satisfies
\[\begin{split}
F(g.z,\overline{g.z},\lambda)&=F_- (g.z,\lambda) F_+(g.z,\overline{g.z},\lambda)\\
&=F_- (z,\lambda) F_+(g.z,\overline{g.z},\lambda)\\
&=F(z,\bar{z},\lambda)F_+(z,\bar{z},\lambda)^{-1}F_+(g.z,\overline{g.z},\lambda).\\
\end{split}\]
The reality of $F(g.z,\overline{g.z},\lambda)$ and $F(z,\bar{z},\lambda)$ yields $F_+(z,\bar{z},\lambda)^{-1}F_+(g.z,\overline{g.z},\lambda)=K(g, z,\bar{z})$. Hence (2) follows.

``(3) $\Rightarrow$ (4)": Assume $C$ is a holomorphic extended frame defined on $M$. Then  $C(g.z,\lambda)=C(z,\lambda)B_+(g,z,\bar{z},\lambda)$
by assumption and $F_-=CS_+$ implies
\[F_-(g.z,\lambda)=C(z,\lambda)B_+(g,z,\bar{z},\lambda)S_+(g.z,\overline{g.z},\lambda)=F_-(z,\lambda)  T_+(g,z,\bar{z},\lambda).\]
But since $F_-=e+O(\lambda^{-1})$, this Birkhoff decomposition is unique and $T_+(g,z,\bar{z},\lambda)=e$ follows.
\end{proof}

\begin{corollary}
 { Let  $M $   be any Riemann surface and $G/K$ any inner symmetric space and let $\mathcal{F}: M\rightarrow G/K$ be any harmonic map  in $G/K$.  Then the normalized extended frame  $F_-$,  is fixed under the action  of $\pi_1(M)$.
 In particular,  we have $F_- (g.z, \lambda) =   F_- (z, \lambda)$ for all $z \in \tilde{M}$ and $g \in \pi_1(M)$.
 As a consequence, the normalized potential  {$\eta_- = F_-(z,\lambda)^{-1} \dd F_-(z,\lambda)$,}  is fixed under the action  of
 each $g \in \pi_1(M)$.
}
\end{corollary}

If we assume $M$ to be non-compact then we obtain the following stronger result:
\begin{theorem}\label{thm-invariant-potential} The following statements are equivalent for harmonic maps
$\mathcal{F}:M\rightarrow G/K$ into the
inner symmetric space $G/K$, compact or non-compact, where  $M$ is any non-compact Riemann surface:
 \begin{enumerate}
\item $\mathcal{F}_{\lambda}: \tilde{M} \rightarrow G/K$ has trivial monodromy for all $\lambda \in S^1$.
\item There exists some extended frame $F$ which satisfies
\[F(g.z,\overline{g.z},\lambda)=F(z,\bar{z},\lambda) k(g, z, \bar z)~~ \hbox{ for all }~~\lambda\in S^1,\ g\in \pi_1(M)
\hbox{ and some } k(g, z, \bar z) \in K.\]
\item There exists a holomorphic extended frame $C$ for the harmonic map $\mathcal{F}$ which satisfies
\[C(g.z,\lambda)=C(z,\lambda)\ \hbox{ for all }\  \lambda\in S^1,\ g\in\pi_1(M).\]
\item The integrated normalized potential $F_-$ of the harmonic map $ \mathcal{F}:M\rightarrow G/K$ satisfies
\[F_-(g.z,\lambda)=F_-(z,\lambda)\ \hbox{ for all }\  \lambda\in S^1,\ g\in\pi_1(M).\]
 \end{enumerate}
 \end{theorem}

\begin{proof}
From Proposition \ref{prop-frame} we know that $(1)$ and $(4)$ are equivalent.

 {
By the proof of part $(1)$ of Theorem \ref{thm-monodr} we know that in general there exists some holomorphic extended frame satisfying  $C(g.z,\lambda)=M(g,\lambda)C(z,\lambda)$ for all $g \in \pi_1(M).$
Therefore,  by using such a $C(z,\lambda)$ in $(3)$  of
Proposition \ref{prop-frame} we observe that the monodromy is trivial if and only if
$C(z,\lambda)$ is invariant under $\pi_1(M).$ Finally, $(2)$ is equivalent to  $(2)$ of
Proposition \ref{prop-frame}.}
\end{proof}

\begin{remark}
It is not known (except for some examples)  for which $M$ and which $G/K$ the factor $k$ in $(2)$ can be removed.
\end{remark}

 { For   Willmore surfaces we consider the inner symmetric space
 $SO^+(1,n+3)/SO^+(1,3)\times SO(n)$. In this case we have in view of Theorem \ref{thm-monodr}:
 \begin{theorem}
 If the inner symmetric space is $SO^+(1,n+3)/SO^+(1,3)\times SO(n),$ then the
  statements of Theorem \ref{thm-invariant-potential} hold for any Riemann surface, compact or non-compact.
 \end{theorem}
 }

\subsubsection{Dressing and trivial monodromy}

Dressing is a very useful operation which permits to construct new harmonic maps from a given one.
Before applying this to finite uniton harmonic maps we
recall the definition of ``dressing" (e.g. \cite{DPW}, see also \cite{Gu-Oh}, \cite{BP}, \cite{Do-Ha3}, \cite{TU1}).

Let $\mathcal{F}:\tilde{M}\rightarrow G/K,$ a harmonic map, and  $F$  an extended frame satisfying
$F(z_0,\bar{z}_0,\lambda) = e$ for some base point $z_0$. Let $h_+\in\Lambda^+G^{\C}_{\sigma}$ and consider the Iwasawa splitting (on an open subset $\tilde{M}'$ of $\tilde{M}$ containing $z_0$)
\begin{equation}
\label{eq-dress1} h_+F=\hat{F}\hat{W}_+.
\end{equation}
Then $\hat{F}(z, \bar z,\lambda)$ defines a  {family of harmonic maps on $\tilde M'$}
  \begin{equation}
\label{eq-dress2}
 {\hat{\mathcal{F}}_\lambda : = h_+ \sharp {\mathcal{F}_\lambda}  : \tilde{M}' \rightarrow G/K,\
\hat{\mathcal{F}}_{\lambda} ( z, \bar z) :=\hat{F} (z, \bar z,\lambda) \mod K . }
\end{equation}

\begin{lemma}
Using the notation introduced just above,  {$\hat{F}(z, \bar z,\lambda)$} is an extended frame of
$\hat{\mathcal{F}}(z, \bar z,\lambda) = { (h_+\sharp\mathcal{F})(z, \bar z, \lambda)}$ and satisfies w.l.g.  {$\hat{F}(z_0, \bar z_0,\lambda) = e$}.
If $M(g,\lambda)$ denotes the monodromy representation of $F(z, \bar z,\lambda)$, $g \in \pi_1(M)$, then the
monodromy representation of $\hat{F}$ is given by
\begin{equation}
 \hat{M}(g,\lambda) =    h_+ M(g,\lambda) h_+^{-1}.
\end{equation}
\end{lemma}

\begin{corollary} \label{Cor1}
If $\mathcal{F}_\lambda$ descends  for some $\lambda_0 \in S^1$ to a harmonic map $\mathcal{F}_{\lambda_0}:M\rightarrow G/K$, i.e. if  {$M(g,\lambda_0) = e$} for all $g \in \pi_1(M)$,
then the corresponding dressed harmonic map $h_+\sharp\mathcal{F}_{\lambda_0}$ will descend to
 $M' = \tilde{M}' \mod \pi_1(M)$.
\end{corollary}

A particularly interesting feature of harmonic maps with trivial monodromy representation
 is that  {the property of trivial monodromy representation is preserved under dressing action}.

\begin{theorem}\label{thm-dress} If $\mathcal{F}:\tilde{M}\rightarrow G/K$ is a harmonic map with
trivial monodromy representation, then the associated family $\mathcal{F}_{\lambda}$ consists of harmonic maps with
trivial monodromy and all dressed harmonic maps $h_+\sharp\mathcal{F}_{\lambda}$, $h_+\in\Lambda^+G^{\C}_{\sigma}$,
have trivial monodromy representations.
\end{theorem}


\subsection{ {Relating harmonic maps into a non-compact inner symmetric
 space $G/K$ to  harmonic maps into  the compact dual  inner symmetric space of $G/K$}}

 {In \cite{BuGu}, Burstall and Guest have given  for compact inner symmetric spaces $G/K$  an explicit
description of those normalized potentials which produce finite uniton type harmonic maps.
To make their work applicable to the non-compact case, we established in \cite{DoWa13} a duality theorem between harmonic maps into non-compact inner symmetric spaces and their compact dual. The most important feature of this result is that the corresponding harmonic maps share the same normalized potential (See the next theorem).}
 {We will show  briefly how this duality relation works and what it implies for the properties $(U1)$
and $(U2)$}.

 {Let $G/K$ be a  non-compact inner symmetric space, defined by
 $\sigma$,  and $\tilde{G}$ the (connected) simply-connected cover  of $G$.
Then $G/K = \tilde{G}/\tilde{K}$ for some closed subgroup $\tilde{K}$ of  $\tilde{G}.$ }
 {Let $\tilde{U},$  be a (connected, simply-connected  and semi-simple)  maximal compact Lie subgroup
of $\tilde{G}^{\C}$, the complexification of $\tilde{G}.$}  We  {can also  assume w.l.g. that   $\tilde{G}^{\C}$  is a complex matrix Lie group with Lie subgroup $G$.
  Moreover, we can assume w.l.g. that     $\tilde{U}$   is invariant under the $\C-$linear extension of $\sigma$ to $G^\C$. Finally, by abuse of notation, for $G = SL(2,\R)$ we also denote by $\tilde{G}$ and $\tilde{K}$ the natural image of  $\tilde{G}$ and $\tilde{K}$ in $G^\C$. See, e.g. \cite{Hoch}.}

  We consider the compact dual $\tilde{U}/(\tilde{U}\cap \tilde{K}^{\C})$  of $G/K$ which
 clearly also is defined by $\sigma$ (see \cite{DoWa13} for more details).
Moreover,  $ (\tilde{U}\cap \tilde{K}^{\C})^\C = \tilde{K}^{\C}$ holds
(see {Theorem 1.1 of \cite{DoWa13}}). From this we infer:
\begin{equation}\label{eq-loop-com-n-com}
    \Lambda \tilde{G}^{\C}_{\sigma}=\Lambda \tilde{U}^{\C}_{ {\sigma}},\    \Lambda_*^- \tilde{G}^{\C}_{\sigma}=\Lambda_*^- \tilde{U}^{\C}_{ {\sigma}},\    \Lambda^+\tilde{ G}^{\C}_{\sigma}=\Lambda ^+ \tilde{U}^{\C}_{ {\sigma}}.
\end{equation}
As a consequence, for any extended framing $F(z,\bar{z},\lambda)$ of  $\mathcal{F}: M \to G/K$, the decomposition
 {\[F(z,\bar{z},\lambda)=F_{\tilde{G,-}} (z,\lambda) F_{\tilde{G},+}(z,\bar{z},\lambda)=F_{\tilde{U},-}(z,\lambda)   F_{\tilde{U},+}(z,\bar{z},\lambda)\]
shows
 \begin{equation}\label{eq-norm-framing}
\begin{split}F_{\tilde{G},-} (z,\lambda)&=F_{\tilde{U},-}(z,\lambda),\\
 \eta&=\lambda^{-1}\eta_{-1}\mathrm{d}z=
(F_{\tilde{G},-})(z,\lambda)^{-1}\mathrm{d}F_{\tilde{G},-}(z,\lambda)=(F_{\tilde{U},-})^{-1}(z,\lambda)\mathrm{d}F_{\tilde{U},-}(z,\lambda).
 \end{split}
 \end{equation}}
\begin{theorem} \label{thm-noncompact}$($\cite{DoWa13}$)$
Let
$G/K = \tilde{G} / \tilde{K} $ be a non-compact inner symmetric space with $ \tilde{G} $ simply-connected.
Let $\tilde{U} / (\tilde{U} \cap \tilde{K}^{\C})$ denote the dual compact symmetric space. Then the space
$ \tilde{U} / ( \tilde{U} \cap \tilde{K}^{\C} )$ is inner and  $\Lambda \tilde{G}_{\sigma}^{\C} =
\Lambda \tilde{U}_{\sigma}^{\C} $ holds.

Let $\mathcal{F}:\tilde{M} \rightarrow G/K = \tilde{G} / \tilde{K} $ be a harmonic map, where $ \tilde{M} $
is a simply-connected Riemann surface. Let  {$F(z,\bar{z},\lambda)$ denote an extended frame
of $ \mathcal{F} $. For $F(z,\bar{z},\lambda)$ define
$ F_ {\tilde{U}} $   via the Iwasawa decomposition $F(z,\bar{z},\lambda)=F_{\tilde{U}}(z,\bar{z},\lambda) S_+(z,\bar{z},\lambda)$, with
$F_{\tilde{U}}(z,\bar{z},\lambda)\in \Lambda \tilde{U}_{\sigma}$, $S_+ (z,\bar{z},\lambda)\in \Lambda^+ \tilde{U}_{\sigma}^{\C}$.
Then  \[\hbox{$\mathcal{F}_{\tilde{U}} :\tilde{M} \rightarrow \tilde{U} / \tilde{U} \cap \tilde{K}^{\C}$,
$\mathcal{F}_{\tilde{U}} \equiv F_{\tilde{U}}(z,\bar{z},\lambda) \mod \tilde{U} \cap \tilde{K}^{\C}$}\]} is a harmonic map  {for each fixed $\lambda\in S^1$}.

Moreover, the harmonic maps $\mathcal{F}$ and $\mathcal{F}_{\tilde{U}}$ have the same normalized potential.
\end{theorem}

As a consequence of the theorem above we obtain (see also Theorem 3.6  of \cite{DoWa13})
\begin{corollary}\label{cor-finite}
Let  {$F(z,\bar{z},\lambda)$ and $F_{\tilde{U}}(z,\bar{z},\lambda)$} be the extended frames defined as above. Then
 \begin{enumerate}
\item
 $F(z,\bar{z},\lambda)$ satisfies $(U1)$ if and only if $F_{\tilde{U}}(z,\bar{z},\lambda)$ satisfies $(U1)$;
\item
  $F(z,\bar{z},\lambda)$ satisfies $(U2)$ if and only if $F_{\tilde{U}}(z,\bar{z},\lambda)$ satisfies $(U2)$.
\end{enumerate}
Therefore $\mathcal{F}$ is of finite uniton type if and only if $\mathcal{F}_{\tilde{U}}$ is of finite uniton type.
\end{corollary}

\begin{proof} Let  {$F(z,\bar{z},\lambda): \tilde{M}\rightarrow \tilde{G} / \tilde{K}$ denote an extended frame of $\mathcal{F}$.
Then we have  $F(z,\bar{z},\lambda)=F_{\tilde{U}}(z,\bar{z},\lambda)S_+(z,\bar{z},\lambda)$} as stated in Theorem \ref{thm-noncompact}. We need to verify
the properties (U1) and (U2) in Definition \ref{def-uni}.

(U2): If  {$F(z,\bar{z},\lambda)$ is a Laurent polynomial in $\lambda$, then $F_{\tilde{U}}(z,\bar{z},\lambda)$ contains only finitely many
negative powers of $\lambda$ since $F_{\tilde U}(z,\bar{z},\lambda)=F(z,\bar{z},\lambda)(S_+(z,\bar{z},\lambda))^{-1}$. Since $F_{\tilde U}(z,\bar{z},\lambda)$ satisfies also a reality
condition, $F_{\tilde U}(z,\bar{z},\lambda)$ is a Laurent polynomial. The converse statement follows by interchanging  $F(z,\bar{z},\lambda)$ and $F_{\tilde{U}}(z,\bar{z},\lambda)$.}

(U1):  We distinguish two cases: The case of $M=S^2$ is trivial, but one needs to recall from Section 3.2 of \cite{DoWa12}  that in this case extended frames have some singular points in general.

 If $M$ is different from $S^2$, then we obtain
  \[F(g.z,\overline{g.z},\lambda)=\chi(g,\lambda)F(z,\bar{z},\lambda) k(g,z,\bar{z}) \ \hbox{ for all }\ g\in\pi_1(M)\]
with $k(g,z,\bar{z})\in K$ and $\chi\in \Lambda {\tilde{G}_{\sigma}}$. Assume now that $F$ satisfies (U1).
Then, without loss of generality $\chi(g,\lambda)=e$ for all $g\in\pi_1(M)$.
Inserting $F=F_{\tilde{U}}S_+$ we obtain
\[F_{\tilde{U}}(g.z,\overline{g.z},\lambda)=F_{\tilde{U}}(z,\bar{z},\lambda)u(g, z,\bar{z}), u\in {\tilde{U}}\cap {\tilde{K}}^{\C},\]
 where $u(g,z,\bar{z}) = S_+ (z, \bar{z},\lambda) k(g,z,\bar{z}) S_+ (g.z, \overline{g.z},\lambda)^{-1}.$
 {Hence both $F_{\tilde{U}}$ and   $\mathcal{F}_{\tilde U}$ satisfies $(U1)$.}

Assume now that $F_{\tilde{U}}$ satisfies $(U1).$
Hence $F_{\tilde{U}}(g.z,\overline{g.z},\lambda)=F_{\tilde{U}}(z,\bar{z},\lambda){u(g,z,\bar{z})}$
for all $g\in\pi_1(M)$
and  with $u \in \Lambda^+U^\C_\sigma = \Lambda G^\C_\sigma$.  Substituting $F_{\tilde{U}} = F (S_+)^{-1}$  from above we obtain
\[
F(g.z,\overline{g.z},\lambda)=F(z,\bar{z},\lambda) L_+(g,z,\bar{z},\lambda),
\]
where  $L_+(g,z,\bar{z},\lambda) = S_+(z,\bar{z},\lambda)^{-1} u(g,z,\bar{z}) S_+(g.z,\overline{g.z},\lambda)$. But then $L_+(g,z,\bar{z},\lambda) = k(g,z,\bar{z}) \in K$ and $F$ satisfies $(U1)$.
Hence $F$ and then also $\mathcal{F}$  satisfy $(U1)$.
\end{proof}

Thus finding the normalized potentials  of the finite uniton type harmonic maps $\mathcal{F}$ into the
non-compact inner symmetric space $G/K$ means finding the normalized potentials for the corresponding
finite uniton type harmonic maps $\mathcal{F}_{\tilde{U}}$ into the compact dual ${\tilde{U}}/({\tilde{U}}\cap {\tilde{K}}^{\C})$.
The latter task has been greatly simplified by \cite{BuGu}.\\


 \subsection{Remarks on monodromy, dressing and some application to Willmore surfaces}

Above we have defined finite uniton harmonic maps by the conditions (U1) and (U2) of Definition \ref{def-uni}.
The condition (U1) is always satisfied if $M$ is simply connected (including $S^2$ as it was remarked
in $(2)$ of Remark \ref{sphere}  that for the case of $M = S^2$, the usual loop group approach
applies, modulo a minor precaution).
 Corollary \ref{cor-finite} yields immediately
\begin{proposition}\label{prop-trivial-monod} With the assumptions and the notations of Theorem \ref{thm-noncompact}
the following statements are equivalent:
 \begin{enumerate}
\item $ \mathcal{F}$ has trivial monodromy  {representation,}

\item $\mathcal{F}_{\tilde{U}}$ has trivial monodromy  {representation.}
 \end{enumerate}
\end{proposition}

\begin{corollary}\label{cor-finite-dress}
If $\mathcal{F}:M\rightarrow G/K$ is a harmonic map of finite uniton type and $h_+\in\Lambda^+G^{\C}_{\sigma}$.
Then the dressed harmonic map $h_+\sharp\mathcal{F}_{\lambda}$ is of finite uniton type on $M$.
\end{corollary}
\begin{proof} Condition (U1) for $h_+\sharp\mathcal{F}_{\lambda}$  follows from Theorem \ref{thm-dress}.
It is easy to verify that
$h_+F=\hat{F}\hat{W}_+$ implies that $\hat{F}$ only contains finitely many negative powers of $\lambda$ if $F$ does.
But $\hat{F}$ is real and the claim follows.  Also note that if $ {F(z_0,\bar{z}_0,\lambda) = e}$ then one can assume
without loss of generality that also $\hat{F}(z_0,\bar{z}_0,\lambda) =  {e}$ holds.
\end{proof}


\section{  {The work of Burstall and Guest} and the  DPW formalism}

In this section, we will recall the work of Burstall and Guest on harmonic maps of finite uniton  {number} needed in this paper and then translate it into the language of  { the DPW method} \cite{DPW}. For more details on Burstall and Guest's work, we refer to  \cite{BuGu} and \cite{Gu2002}. Note that, in their work, $G$ is assumed to be a connected, compact, semi-simple real Lie group with trivial center,  {as it is in this section, }
and  $\mathfrak{g}$  denotes its Lie algebra and $G^{\mathbb{C}}$ its complexification.
 {We can assume w.l.g. that $G^\C$ is a semisimple simply-connected matrix Lie group and $G$ a subgroup of $G^\C$ \cite{Hoch}.}  {Note also that \cite{BuGu} only considers inner symmetric spaces, as does this paper.}

\subsection{Review of extended solutions}
 In this subsection, we  compare/unify the notation used in \cite{Uh}, \cite{BuGu} and \cite{DPW}.
For a harmonic map  $\mathbb{F}:\D\rightarrow G$, in \cite{Uh},  \cite{BuGu} ``extended solutions" are considered, while in \cite{DPW} always ``extended frames'' are used.  In this subsection we will explain the relation between these methods. Here we include primarily  the details which we will need to use.

For the convenience of the reader and to fix notation we start with a simple remark.

\subsubsection{Inner symmetric spaces and the modified Cartan embedding}

Consider the inner compact symmetric space  $G/\hat{K}$  with inner involution $\sigma$, given by $\sigma(g) = h g h^{-1}$ and with $\hat{K} = Fix^\sigma (G)$. Note that $h^2 \in Center (G) = \{e\}. $  Then with $R_h(g)= gh$ we consider
the map
\[
\begin{tikzcd}[column sep=6mm,row sep=4mm]
G/\hat{K}  \ar{r}{\mathfrak{C}}    &   G     \ar{r}{R_h} & G \\
g   \ar{r}{}  &   g \sigma(g)^{-1} = ghg^{-1} h^{-1}  \ar{r}{} & \mathfrak{C}_h:=g h g^{-1}.
\end{tikzcd}
\]
 In this way $G/\hat K$ is an isometric, totally geodesic submanifold   of $G$ \cite{BuGu}, and ${\mathfrak{C}_h}$ will be called the ``modified Cartan embedding". Note that for outer symmetric spaces the above Cartan embedding does not apply directly.

\subsubsection{Extended frames for harmonic maps  $\mathcal{F} : M \rightarrow G/\hat{K}$ and modified  harmonic maps
$\mathfrak{C}_h\circ \mathcal{F}$}
Using the notation introduced above, consider a harmonic map  $\mathcal{F}: M \rightarrow G/\hat{K}$.

By $F : \tilde{M} \rightarrow  \Lambda G_{\sigma}$ we  denote the extended frame of $\mathcal{F}$ which is normalized to $F(z_0, \bar{z}_0, \lambda) = e$ at some base point $z = z_0$ for all $\lambda \in S^1$.
The extended frame of a harmonic map $\mathcal{F}$ actually is for each fixed $\lambda$ the frame of the corresponding immersion $\mathcal{F}_\lambda$ of the
associated family of $\mathcal{F}$. Obviously, the twisting condition in our case means
\[{\sigma(\gamma)(\lambda) = h \gamma(-\lambda)h^{-1} = \gamma(\lambda)}\hbox{
for all
$\gamma \in \Lambda G_{\sigma}$.}\]

Next we consider the composition of the family of harmonic maps  $\mathcal{F}_\lambda$ with the
modified Cartan embedding $\mathfrak{C}_h$.  In our setting, since
$\mathcal{F}_\lambda =F( z, \bar{z}, \lambda)\mod \hat{K}$,   this yields the
$\lambda-$dependent harmonic map $\mathfrak{F}^h_\lambda$ given by
 \begin{equation} \label{harmonicrelation}
  \mathfrak{F}^h_{\lambda}=
 F( z, \bar{z}, \lambda) h F( z, \bar{z}, \lambda))^{-1}.
 \end{equation}
Note that $\mathfrak{F}^h_\lambda $  is a $\lambda-$dependent harmonic map satisfying
$(\mathfrak{F}^h_\lambda )^2=e, $ where the square denotes the product in the group $G$.
Moreover, we also  have     $\mathfrak{F}^h_\lambda (z_0, \bar{z}_0,\lambda)  = h.$ Harmonic maps into $G$ satisfying these two properties will be called ``modified harmonic maps".

\begin{theorem}
We retain the notation and the assumptions made just above. In particular, $z_0$ is a fixed basepoint in the Riemann surface $M$ Then
there is a bijection between harmonic maps  $\mathcal{F} : M \rightarrow G/\hat{K}$  satisfying
$\mathcal{F}(z_0,\bar{z}_0,\lambda) = e\hat{K}$ and modified harmonic maps, i.e. harmonic maps
  $\mathfrak{F}^h_{\lambda}: M \rightarrow G_h = \{ ghg^{-1}; g \in G \}$ satisfying
   $(\mathfrak{F}^h_{\lambda})^2 = e$ and  $\mathfrak{F}^h_{\lambda}(z_0, \bar{z}_0, \lambda) = h$ for all   $\lambda \in \C^*$.
  This relation is given by composition with the modified  Cartan embedding (and its inverse respectively).
  \end{theorem}

\begin{proof}
We have shown ``$\Longrightarrow$" above. Assume now we have  a harmonic map
$\mathbb{F}: M \rightarrow  {{G}_h}$ satisfying
 $\mathbb{F}^2 = e$ and  $\mathbb{F}(z_0, \bar{z}_0, \lambda) = h$ for all  $\lambda \in \C^*$.
   Since   $\mathfrak{C}_h$ is an isometric  diffeomorphism   onto its image, we consider
    (for all $\lambda \in \C^*$) the $\lambda$-dependent harmonic map
    $\mathcal{F}_\lambda = (\mathfrak{C}_h)^{-1} \circ \mathbb{F}_\lambda : M \rightarrow G/\hat{K}$.
    Clearly, then we have  $\mathcal{F}(z_0, \bar z_0, \lambda) = e\hat{K}$.  Moreover, by what was shown above, we now infer
    $\mathbb{F}(z, \bar z, \lambda) = \mathfrak{F}^h_{\lambda} (z, \bar z, \lambda) =
    F( z, \bar{z}, \lambda) h F( z, \bar{z}, \lambda))^{-1}$, where $F$ denotes the extended frame of $\mathcal{F}$.
    Since     $\mathcal{F} (z_0, \bar{z}_0, \lambda) = e$,
    $h =   F( z_0 \bar{z}_0, \lambda) h F( z_0, \bar{z}_0, \lambda))^{-1}= \mathfrak{F}^h_{\lambda} ( z_0, \bar{z}_0, \lambda)$ holds.
    \end{proof}


\subsubsection{Extended solutions for harmonic maps into Lie groups}

We  have shown, among other things, in the  subsubsections above  that it is essentially sufficient
for our purposes to  consider  harmonic maps into Lie groups $G$.
In this subsubsection we consider harmonic maps into Lie groups following the approach of
\cite{Uh} and \cite{BuGu}.
 We start by relating the different loop parameters
used in \cite{Uh} , \cite{BuGu} and \cite{DPW} respectively to each other.

 To begin with, we  recall the definition of {\em extended solutions}
following Uhlenbeck \cite{Uh},\cite{BuGu}. Let  $\D \subset \C$ be a
simply-connected domain and $\mathbb{F}: \D\ \rightarrow G$  a harmonic map.
Set
\[\mathbb{A}=\frac{1}{2} \mathbb{F}^{-1}\mathrm{d} \mathbb{F} =\mathbb{A}^{(1,0)}+\mathbb{A}^{(0,1)}.\]
Consider {for $\tilde{\lambda} \in \C^*$} the equations
\begin{equation}\label{eq-Uh1}
\left\{
\begin{split}
\partial_z \Phi \mathrm{d}z&=(1-\tilde{\lambda}^{-1})\Phi\mathbb{A}^{(1,0)},\\
\partial_{\bar{z}} \Phi \mathrm{d} \bar{z}&=(1-\tilde{\lambda})\Phi\mathbb{A}^{(0,1)}.\\
\end{split}
\right.
\end{equation}
with $\Phi: \D \rightarrow \Omega G$, where the corresponding loop parameter
 is denoted here by $\tilde{\lambda}.$
 Then, by Theorem 2.2 of \cite{Uh} (Theorem 1.1 of \cite{BuGu}),
there exists a solution $\Phi(z,\bar{z},\tilde{\lambda})$ to the above
equations such that
\begin{equation}
\Phi(z,\bar{z},\tilde{\lambda}=1)=e,~\ \mbox{and}
~ \Phi(z,\bar{z},\tilde{\lambda}=-1)= \mathbb{F} (z, \bar z)
\end{equation}
hold. This solution is unique up to multiplication by some
$\gamma \in \Omega G = \{ g \in \Lambda G^{\C}_\sigma |,  g(\lambda = 1) = I \}$ satisfying $\gamma (-1) = e$. Such solutions $\Phi$ are said to be {\bf extended solutions}.

   If we also have $\mathbb{F}(z_0)=e$ , we can also choose $\Phi(z_0,\bar{z}_0, \tilde\lambda)=e$.
Although the assumption  $\mathbb{F}(z_0)=e$ was
used in \cite{Uh}, we will, as in \cite{BuGu}, not assume  this, since it is not satisfied
 in a large part of this section. The following statement  is straightforward.

\begin{lemma}\label{lemma-es} \cite{BuGu}  Let $\Phi(z,\bar{z}, \tilde{\lambda})$ be
 {\em an
extended solution} of the harmonic map  $\mathbb{F}:\D\rightarrow G$. Let $\gamma\in \Omega G$.
Then $\gamma(\lambda)\Phi(z,\bar{z},\tilde{\lambda})$ is an extended solution of the
harmonic map $\gamma(-1)\mathbb{F}(z,\bar{z})$.
\end{lemma}

\begin{remark}
	Next we show how the ``DPW approach" without the basepoint assumption  \cite{DPW} naturally leads to
	Uhlenbeck's extended solutions \cite{Uh}.
	We follow Section 9 of \cite{Do-Es} and consider the Lie group $G$  as the outer symmetric space
	$G = (G \times G)/ \Delta,$
where the defining symmetry $\tilde{\sigma}$ is given by
	$$\tilde{\sigma} (a,b) = ( b,a)$$ and we have  $$\Delta  = \{ (a,b) \in  G \times G; a=b \}.$$
	
	For the purposes of our loop group method it is necessary to consider the $G \times G-$loop group
	$\Lambda(G \times G)$
	twisted by $\tilde{\sigma}$. We thus consider the automorphism of the loop group
	$\Lambda(G \times G) = \Lambda G \times \Lambda G$ given by
	$$\hat{\tilde{\sigma}} ((a,b)) (\lambda) = \tilde{\sigma}( a(\lambda), b(\lambda)) =
	(b(\lambda), a(\lambda)).$$

It is straightforward to verify that the twisted loop group $\Lambda(G \times G)_{\tilde{\sigma}}$
	is given by
	$$\Lambda(G \times G)_{\tilde{\sigma}} = \{ (g(-\lambda), g(\lambda)) ; g(\lambda) \in \Lambda G \}
	 \cong \Lambda G. $$

	Let's consider now a harmonic map $\mathbb{F}: M \rightarrow G.$
	Then the map
	$\mathfrak{F}:M \rightarrow G \times G$, given by
	 $\mathfrak{F}(z, \bar{z}) = \left(\mathbb F(z,\bar{z}),e\right)$,  is a global frame of $\mathbb{F}$.
	
	Following \cite{DPW} one needs to decompose the Maurer-Cartan form
	$\mathfrak{A} = \mathfrak{F}^{-1} \mathrm{d} \mathfrak{F}$ of $\mathfrak{F}$  into the eigenspaces of
	$\tilde\sigma$ and to introduce the loop parameter $\lambda$.
	One obtains (see \cite{Do-Es}, formula $(68)$):
	\begin{equation}
	\mathfrak{A}_{\lambda} = \left( \left(1+\lambda^{-1}\right) \mathbb A^{(1,0)} +  \left(1+\lambda\right)\mathbb A^{(0,1)} ,\ \left(1-\lambda^{-1}\right) \mathbb A^{(1,0)}  +  \left(1-\lambda\right) \mathbb A^{(0,1)} \right).
	\end{equation}
\end{remark}
	\begin{theorem}
		Let  $G$  be a connected, compact or non-compact, semi-simple real Lie group with trivial center.
		Let  $\mathbb{F}:\D\rightarrow G$ be a harmonic map.  Then, when  representing $G$ as the
		symmetric space  $G = (G \times G)/\Delta,$
		any extended  frame $\mathfrak{F} : \D \rightarrow \Lambda(G \times G)_{\tilde\sigma}$
		of $\mathbb{F}$ satisfying
		$\mathfrak{F}(z, \bar{z}, \lambda = 1) = \left(\mathbb F(z,\bar{z}),e\right)$
		is given  by a pair of functions,
		\[\mathfrak{F}(z, \bar{z},\lambda) =  ( \Phi (z, \bar{z},-\lambda), \Phi (z, \bar{z},\lambda)),\]
		where the  matrix function $ \Phi (z, \bar{z},\lambda)$ is an  extended solution  for $\mathbb{F}$ in the sense of Uhlenbeck \cite{Uh} as introduced above.
	\end{theorem}
	
	\begin{proof}
		Since the two components of $\mathfrak{A} $ only differ by a minus sign in $\lambda$, any solution to
		the equation $\mathfrak{A}_\lambda  = \mathfrak{F}(z, \bar z, \lambda)^{-1} \mathrm{d} \mathfrak{F}(z, \bar z, \lambda)$
		is of the form $ \mathfrak{F}(z, \bar z, \lambda) =
		(B(-\lambda) \Psi(z, \bar z, - \lambda) , B(\lambda) \Psi(z, \bar z,  \lambda))$,
		where $\Psi$ solves the equations \eqref{eq-Uh1}. Moreover, we can assume w.l.g. that $\Psi$ satisfies the two conditions for extended solutions stated above for $\lambda = \pm1$.
		Now
		$\mathfrak{F}(z, \bar{z}, \lambda = 1) = \left(\mathbb F(z,\bar{z}),e\right)$ implies
		{$B(1)  = B(1) \Psi (z, \bar{z},\lambda= 1) =e$ and $B( -1)\Psi(z, \bar z, -1) =
		 B(-1) \mathbb{F}(z, \bar z) = \mathbb{F}(z, \bar z),$
		 whence $B(-1) = e$ follows.
		 Therefore, $\Phi (z, \bar z, \lambda)  = B(\lambda)  \Psi(z, \bar z, \lambda)$}
		yields the claim.
	\end{proof}

	Note that in this theorem no normalization is required. Moreover, the loop parameter used in \cite{Uh} is the same as the one used in \cite{DPW}. However, the matrix functions $B(\lambda)$ and  $\Phi (z, \bar z, \lambda)$
	are not uniquely determined which causes the DPW procedure to yield quite arbitrary potentials, not easily permitting any converse construction procedure.

\subsubsection{Extended solutions and extended frames for harmonic maps into symmetric spaces} \label{414}

Consider as before a harmonic map  $\mathcal{F}: M \rightarrow G/\hat{K}$ into a symmetric space
with inner involution $\sigma$, given by $\sigma(g) = h g h^{-1}$ and with $\hat{K} = Fix^\sigma (G)$.
 As above we consider the modified harmonic map $\mathbb{F}: \rightarrow G$ given by
 $$\mathbb{F}(z, \bar z, \lambda) =   \mathfrak{F}^h_{\lambda} (z, \bar z, \lambda) =
 F( z, \bar{z}, \lambda) h F( z, \bar{z}, \lambda))^{-1}.$$
  For this $\lambda-$family of harmonic maps $\mathbb{F}_\lambda$ we compute
 \begin{equation*}
 \begin{split}
\mathbb{A}&=\frac{1}{2}\mathbb{F}_{\lambda}^{-1}\mathrm{d}\mathbb{F}_{\lambda} \\
\ &=\frac{1}{2}\left(F(z,\bar{z},\lambda)hF(z,\bar{z},\lambda)^{-1}\right)^{-1}\mathrm{d}\left(F(z,\bar{z},\lambda)hF(z,\bar{z},\lambda)^{-1}\right)\\
\ &=\frac{1}{2}F(z,\bar{z},\lambda) h^{-1}\alpha_{\lambda} h F(z,\bar{z},\lambda)^{-1}-\frac{1}{2}\mathrm{d}F(z,\bar{z},\lambda)F(z,\bar{z},\lambda)^{-1}\\
\ &=\frac{1}{2}F(z,\bar{z},\lambda)\left(\alpha_{-\lambda}-\alpha_{\lambda}\right)F(z,\bar{z},\lambda)^{-1}\\
\ &=-F(z,\bar{z},\lambda)\left(\lambda^{-1}\alpha_{\mathfrak{p}}'+\lambda\alpha_{\mathfrak{p}}''\right)F(z,\bar{z},\lambda)^{-1}.\\
\end{split}
\end{equation*}
Following Uhlenbeck's approach we need to introduce a new ``loop parameter" $\tilde{\lambda}$ now and consider Uhlenbeck's differential equation \eqref{eq-Uh1} for
$\Phi(z,\bar{z},\lambda,\tilde{\lambda})$ with conditions for $\tilde{\lambda} = \pm 1$:
\begin{equation}\label{eq-Uh2}
\left\{
\begin{split}
\partial_z \Phi \mathrm{d}z&=-\Phi(1-\tilde{\lambda}^{-1})\lambda^{-1}F(z,\bar{z},\lambda) \alpha_{\mathfrak{p}}' F(z,\bar{z},\lambda)^{-1}\\
\partial_{\bar{z}} \Phi\mathrm{d}\bar z&=-\Phi(1-\tilde{\lambda})\lambda F(z,\bar{z},\lambda) \alpha_{\mathfrak{p}}'' F(z,\bar{z},\lambda)^{-1}\\
\end{split}
\right.\end{equation}
From (\ref{eq-Uh2}) it is natural to consider the Maurer-Cartan form $\widetilde{\mathbb{A}} $ of $\Phi (z,\bar{z},\lambda,\tilde{\lambda})F(z,\bar{z},\lambda)$. One obtains:
\begin{equation}
\begin{split}
\widetilde{\mathbb{A}}&=(\Phi F)^{-1} \dd ( \Phi F)\\
&=F^{-1}\dd F+F^{-1}(\Phi^{-1}\dd  \Phi) F\\
&=\alpha_{\lambda}-F^{-1}\Phi^{-1}\left(\Phi(1-\tilde{\lambda}^{-1})\lambda^{-1}F  \alpha_{\mathfrak{p}}' F ^{-1}+\Phi(1-\tilde{\lambda})\lambda F \alpha_{\mathfrak{p}}'' F ^{-1}\right)F\\
&=\alpha_{\lambda}- \left( (1-\tilde{\lambda}^{-1})\lambda^{-1}  \alpha_{\mathfrak{p}}'  +(1-\tilde{\lambda})\lambda \alpha_{\mathfrak{p}}''\right)\\
&=\lambda^{-1}  \alpha_{\mathfrak{p}}'+\alpha_{\mathfrak k}+\lambda \alpha_{\mathfrak{p}}''- \left( (1-\tilde{\lambda}^{-1})\lambda^{-1}  \alpha_{\mathfrak{p}}'  +(1-\tilde{\lambda})\lambda \alpha_{\mathfrak{p}}''\right)\\
&=\tilde{\lambda}^{-1}\lambda^{-1}  \alpha_{\mathfrak{p}}'+\alpha_{\mathfrak k}+\tilde{\lambda}\lambda \alpha_{\mathfrak{p}}'' \\
&=\alpha_{\tilde{\lambda}\lambda}. \\
\end{split}
\end{equation}
From this we derive immediately the relation
\begin{equation}\label{eq-lawson}
F(z,\bar{z}, \lambda\tilde{\lambda})=A(\lambda, \tilde{\lambda}) \Phi (z,\bar{z},\lambda,\tilde{\lambda})F(z,\bar{z}, \lambda).
\end{equation}
Substituting here  $z = z_0$ we derive, in view of the normalization of $F$ at $z = z_0$:
\begin{equation}\label{eq-A}
A(\lambda, \tilde{\lambda}) = \Phi (z_0 ,{\bar{z}}_0,\lambda,\tilde{\lambda})^{-1}.
\end{equation}
In particular, setting $\lambda=1$ in \eqref{eq-lawson} we obtain
\begin{equation}
F(z,\bar{z}, \tilde{\lambda})=A(1, \tilde{\lambda}) \Phi (z,\bar{z},1,\tilde{\lambda})F(z,\bar{z},1).
\end{equation}
Setting $\tilde\lambda=-1$ in \eqref{eq-lawson} we obtain (by using the twisting condition for $F$):
\begin{equation}
A(\lambda, -1)=F(z,\bar{z},-\lambda)\left( \Phi (z,\bar{z},\lambda,-1)F(z,\bar{z},\lambda)\right)^{-1}=\violet{h}.
\end{equation}
Hence
\[\Phi (z, \bar{z},1,-1)={A(1,-1)^{-1} } F(z,\bar{z}, -1)F(z,\bar{z}, 1)^{-1}=F(z,\bar{z}, 1)hF(z,\bar{z}, 1)^{-1}=\mathbb{F}(z,\bar z,1).\]

In summary we obtain (by setting $\lambda = 1$ and replacing $\tilde{\lambda}$ by $\lambda$):
\begin{corollary}\label{cor-Phi-F}
The extended solution $\Phi$, and the $\sigma-$twisted  extended frame $F$ satisfy
\begin{equation}\label{eq-lawson2}
\Phi (z, \bar{z},1,\lambda) = A(1, \lambda)^{-1}F(z,\bar{z},  \lambda)F(z,\bar{z}, 1)^{-1}.
\end{equation}
In particular, $\Phi  (z, \bar{z},1,\lambda)$ is contained in the based loop group $\Omega G$. Moreover,
for $\lambda = -1$  we obtain the harmonic map $ \mathbb{F}(z,\bar z,1) =  \mathbb{F}(z,\bar z) $.
\end{corollary}

\subsection{Finite  {uniton type   \`{a}}
 la Burstall-Guest for harmonic maps into compact Lie groups}

 Let us recall that in Definition \ref{def-uni} we have given the definition of harmonic maps of finite uniton type into Lie groups. Now we want to define the notion of a ``finite uniton number''.
  {It has been introduced by Uhlenbeck \cite{Uh} for $U(n)$ and by Burstall-Guest
 \cite{BuGu} for a general compact real Lie group $G$.}

\begin{definition}  \label{def-f.u.}
Let $\mathbb{F}:M \rightarrow G$  be a harmonic map into a real Lie group $G$.
Assume there exists a global extended
solution $\Phi(z,
\bar{z},\tilde\lambda):M\rightarrow \Lambda G^{\mathbb{C}}$ (i.e., $\mathbb{F}$ has trivial monodromy).
We say that $\mathbb{F}$ has {\it finite uniton number $k$} if
(see \eqref{eq-alg-loop} for the definition of $ \Omega^k_{alg} G $)
 \begin{equation}\Phi(M)\subset \Omega^k_{alg} G ,\
\hbox{ and } \Phi(M)\nsubseteq \Omega^{k-1}_{alg} G .\end{equation}
 In this case we write  $r(\Phi)=k$ and the minimal uniton number of $\mathbb{F}$ is defined
as \[r(\mathbb{F}):=min\{r(\gamma  Ad(\Phi))| \gamma\in \Omega_{alg} Ad G \}.\]
\end{definition}

 {
The notion of finite uniton number harmonic maps is related to extended solutions. In this paper we usually use extended framings of harmonic maps and read off the notion of ``finite uniton type" from these extended frames. It is important to this paper that these two notions describe the same class of harmonic maps.}

\begin{proposition} \label{typeequivnumber}
$\mathcal F$ is a harmonic map of finite uniton type in $G/K$ if and only if
 {$\mathbb{F} = \mathfrak{C} \circ \mathcal{F} \cdot h = \mathfrak{C}_h \circ \mathcal{F}$  given in  section 4.1.1, where $\mathfrak{C}_h$ is the modified  Cartan }embedding of $G/K$ into $G$, is a harmonic  map of finite uniton number.
\end{proposition}

\begin{proof}
$``\Rightarrow"$ By definition, $F(z,\bar z, \lambda\tilde\lambda)$ is a Laurent polynomial in $\lambda \tilde{\lambda}$ and hence also in $\tilde\lambda$. As a consequence of this and
\eqref{eq-lawson} we have that \[A(\lambda,\tilde\lambda)\Phi(z,\bar z, \lambda,\tilde\lambda)=F(z,\bar{z}, \lambda\tilde\lambda)F(z,\bar{z}, \lambda)^{-1}\]
is also a Laurent polynomial in $\tilde\lambda$.  Hence, by  the above definition
 $\mathbb{F}(z,\bar z,\lambda)$ is of finite uniton number.

 $``\Leftarrow"$  By definition, there exists some $\Phi(z,\bar z, \lambda,\tilde\lambda)$ being a Laurent polynomial in $\tilde\lambda$ and hence by \eqref{eq-A} $A(\lambda,\tilde\lambda)=\Phi (z_* ,{\bar{z}}_*,\lambda,\tilde{\lambda})^{-1}$ is also a Laurent polynomial in $\tilde\lambda$. As a consequence, by
\eqref{eq-lawson} we have that $F(z,\bar{z}, \lambda\tilde\lambda)$ is also a Laurent polynomial in $\tilde\lambda$ and hence also in $\lambda$, i.e., $\mathcal F$ is of finite uniton type.
 Consequently, $\mathbb{F}(z,\bar z,\lambda)$ is of finite uniton number.
  \end{proof}

\begin{remark}\
\begin{enumerate}
\item One of the main goals of this paper is a characterization of the normalized potentials
of all finite uniton type harmonic maps into $G/K$.
The potential for such a  harmonic map is the same as the potential
for the induced harmonic map into $G/\hat{K}$,
where $\hat{K} = G^{\sigma}$, since the different Cartan maps have the same images $\mathfrak{C}(gK)=\mathfrak{C}(g\hat{K})$ for all $g\in G$. We will therefore always assume in this section that
$\hat{K} = G^{\sigma}$. In this case the Cartan map $\mathfrak{C}$ actually is an embedding.

\item   {If $M$ is simply connected, then a global extended solution
$\Phi(z,\bar{z},\lambda):M\rightarrow \Omega G $ always exists, including the case $M=S^2$ (see Theorem 2.2 of \cite{Uh}, also see \cite{Segal}, Theorem 1.1 of \cite{BuGu})}.
This is in contrast to the case of extended
 frames, in which case we have explained above that on $S^2$ the extended frame needs to have (two)
 singularities due to the topology of $S^2$. In general, the extended solution may not exist globally on $M$ if $M$ is not simply connected.
\end{enumerate}
\end{remark}
\vspace{2mm}
Now let us turn to the  Burstall-Guest theory for harmonic maps into Lie groups of  {finite uniton number}.
 Let $\mathrm{T}\subset G $  be a maximal  {torus of $G$}  with $\mathfrak{t}$ the Lie algebra of $\mathrm{T}$.
 We can identify all the homomorphisms from $S^1$ to $\mathrm{T}$ with the integer lattice $\mathcal{I}:=(2\pi)^{-1}\exp^{-1}(e)\cap\mathfrak{t}$  in $\mathfrak{t}$ via the map
\begin{equation}
\begin{array}{llllll}
\mathcal{I}=(2\pi)^{-1}\exp^{-1}(e)\cap\mathfrak{t}&\longrightarrow \hbox{\{homomorpisms from $S^1$ to $\mathrm{T}$\}}, \\
\ \ \ \ \ \ \ \ \ \ \ \ \xi  \ \ \ &\longmapsto \gamma_{\xi},\\
\end{array}
\end{equation}
where  $\gamma_{\xi}:S^1\rightarrow T$ is defined by
\begin{equation}\gamma_{\xi}(\lambda):=\exp(t\xi),\ \ \hbox{ for all }\ \ \lambda=e^{it}\in S^1.
\end{equation}
Let  {$\mathcal {C}_0$} be a fundamental Weyl chamber of $\mathfrak{t}$. Set $\mathcal{I}'=\mathcal {C}_0\cap \mathcal{I}$. Then $\mathcal{I}'$ parameterizes the
conjugacy classes of homomorphisms $S^1\rightarrow G$.
Let $\Delta$ be the set of roots of $\mathfrak{g}^{\mathbb{C}}$. We have the root space decomposition
$\mathfrak{g}^{\mathbb{C}}=\mathfrak{t}^{\mathbb{C}}\oplus (\underset{\theta\in\Delta}{\oplus}\mathfrak{g}_{\theta} ).$  Decompose $\Delta$ as $\Delta=\Delta^-\cup\Delta^+$ according to  {$\mathcal {C}_0$}.
 Let $\theta_1,\cdots,\theta_l\in \Delta^{+}$ be the simple roots. We denote by $\xi_1,\cdots,\xi_l\in \mathfrak{t}$ the basis of $\mathfrak{t}$ which is dual to
$\theta_1,\cdots,\theta_l$ in the sense that $\theta_j(\xi_k)=\sqrt{-1}\delta_{jk}$.

\begin{definition} $($  {p.555 of \cite{BuGu}}$)$
An element $\xi$ in $\mathcal{I}'\backslash\{0\}$ is called a {\em canonical} element, if $\xi=\xi_{j_1}+\cdots+\xi_{j_k}$ with $\xi_{j_1},\cdots,\xi_{j_k}\in \{\xi_1,\cdots,\xi_l\}$ pairwise different. In other words, for every simple root $\theta_j$, we have that $\theta_{j}(\xi)$ only attains the values $0$ or $\sqrt{-1}$.\end{definition}

For $\theta\in\Delta$ and $X\in\mathfrak{g}_{\theta}$ we obtain
$$ad\xi X=\theta(\xi)X\ \ \hbox{ and }\ \theta(\xi)\in\sqrt{-1}\mathbb{Z}.$$
Let $\mathfrak{g}^{\xi}_{j}$ be the $\sqrt{-1}\cdot j-\hbox{eigenspace}$ of $\hbox{ad}\xi$. Then
 \begin{equation} \label{defgjxi}
\mathfrak{g}^{\xi}_j=\underset{\theta(\xi)=\sqrt{-1} j}{\oplus}\mathfrak{g}_{\theta}, \hbox{ and } \mathfrak{g}^{\mathbb{C}}=\underset{j}{\oplus}\ \mathfrak{g}^{\xi}_j.
\end{equation}
We define the {\em height} of $\xi$ as the non-negative integer \begin{equation}r(\xi)=\hbox{ max}\{j|\ \mathfrak{g}^{\xi}_j\neq0\ \}.\end{equation}

\begin{lemma}\label{lemma-xi} (Lemma 3.4 of \cite{BuGu})  Let $\xi=\sum_{i=1}^kn_{j_i}\xi_{j_i}\in \mathcal{I}'$ with $n_{j_i}>0$. Set $\xi_{can}=\sum_{i=1}^k\xi_{j_i}$. Then we have
 \begin{equation}
 \mathfrak{g}^{\xi}_0=\mathfrak{g}^{\xi_{can}}_0,\ \ \ ~~ \sum_{0\leq j\leq r(\xi)-1} \mathfrak{g}^{\xi}_{j+1}=\sum_{0\leq j\leq r(\xi_{can})-1}  \mathfrak{g}^{\xi_{can}}_{j+1}.
 \end{equation}
\end{lemma}
Set
\begin{equation}\left\{
\begin{array}{llllll}
&\mathfrak{f}^{\xi}_j&:= &\underset{k\leq j}{\oplus}\mathfrak{g}^{\xi}_k,\\
 &(\mathfrak{f}^{\xi}_j)^{\perp}&:=&\underset{j< k \leq r(\xi)}{\oplus} \mathfrak{g}^{\xi}_k, \\
 &\mathfrak{u}^0_{\xi}&:=&\underset{0\leq j <r(\xi)} {\oplus}\lambda^{j}(\mathfrak{f}^{\xi}_j)^{\perp}\in \Lambda^+\mathfrak{g}^\C_{\sigma}.\\
\end{array}\right.
\end{equation}
 {Now we can state (in our notation) some of the main results of \cite{BuGu}.}

\begin{theorem}\label{thm-finite-uniton0} (Theorem 1.2, Theorem 4.5, and  {p.560} of \cite{BuGu})
Assume $G$ is connected, compact, and semisimple with trivial center.
\begin{enumerate}
  \item Let $\Phi:M\rightarrow \Omega^k_{alg}G$ be an extended solution of finite uniton number. Then there exists some canonical $\xi\in \mathcal{I}'$, some $\gamma\in \Omega_{alg}G$, and some discrete subset $D'\subset M$, such that on $M\setminus D'$,  {the following Iwasawa decomposition of $\exp C \cdot \gamma_{\xi}$ holds:}
\begin{equation}\label{eq-finite}
 {\gamma\Phi=\exp C \cdot \gamma_{\xi} \cdot (\Phi^+_{\xi})^{-1},}
\end{equation}
where $C:M\rightarrow \mathfrak{u}^0_{\xi}$ is a (vector-valued) {\em meromorphic  {function}} with poles in $D'$
and  {$\Phi^+_{\xi}: M \backslash D' \rightarrow \Lambda^+G^\C_{\sigma}$. Moreover, the Maurer-Cartan form of $\exp C$ is given  by}
\begin{equation}\label{eq-C}
 {(\exp C)^{-1}\dd(\exp C)=\sum_{0\leq j\leq r(\xi)-1}\lambda^j A_j'\dd z,}
\end{equation}
 {with $A_j':M\rightarrow \mathfrak{g}^{\xi}_{j+1}$ being meromorphic functions with poles contained in $D'$  for each $j$.}

 \item
 Conversely, let $\xi\in \mathcal{I}'$ be a canonical element and  $C:M \rightarrow \mathfrak{u}^0_{\xi}$  a meromorphic function satisfying \eqref{eq-C}. Let $D'\subset M$ be the set of poles of $C$,
 {and  $\Phi=(\exp C \cdot \gamma_{\xi}) \cdot (\Phi^+_{\xi})^{-1} $  be an Iwasawa decomposition of $ \exp C \cdot \gamma_{\xi}$. Then $\Phi$ is an extended solution of finite uniton number on $M$.}
\end{enumerate}
\end{theorem}

\begin{remark}\
  \begin{enumerate}
\item
 {Let $z_0\in M\setminus D'$ be some point. Then $C_0=C(z_0)\in \mathfrak{u}^0_{\xi}$
 and $\Phi(z_0,\bar{z}_0,\lambda)\in\Omega^k_{alg}G$ gives some initial condition.}
  \item By Lemma \ref{lemma-es}, since $\Phi(z,\bar{z},\lambda)$ is an extended solution of a harmonic map
  $\Phi(z,\bar{z},-1)$, then  $\gamma\Phi(z,\bar{z},\lambda)$ is  an extended solution of the harmonic map
   $\gamma(-1)\Phi(z,\bar{z},-1)$, which is congruent to the harmonic map $\Phi(z,\bar{z},-1)$ up to the group action of $G$ from the left.  {So for the harmonic map $\gamma(-1)\Phi(z,\bar{z},-1)$ we have an extended solution for which without loss of generality we can always assume $\gamma=e$ in $(1)$ of the theorem above.}

  \item
  On page 560 of \cite{BuGu}, the elements $A'_j$ are defined so that they take values in $\mathfrak{f}^{\xi}_{j+1}=\sum_{k\leq j+1}\mathfrak{g}^{\xi}_{k}$ which is due to the harmonicity of $\mathbb{F}$. In view of the restriction on $C$ to take values in $\mathfrak{u}_{\xi}^0$,  the computations on page 561 of \cite{BuGu} for $(\exp C)^{-1}(\exp C)_z$  imply that
$A'_j$  takes values in $\sum_{k\geq j+1}\mathfrak{g}^{\xi}_{k}$.  Therefore $A'_j \in \mathfrak{g}^{\xi}_{j+1}$ as stated above.
\end{enumerate}
\end{remark}


\subsection{Finite uniton  {type \`{a} la} Burstall-Guest for harmonic maps into symmetric spaces} \label{f.u.alaBuGu}

 {Consider a harmonic map $\mathcal{F}: M \rightarrow G/K$ from $M$ into the inner symmetric space
$G/\hat K$. As stated in \cite{BuGu}, the inner symmetric space $G/\hat{K}$,  can be embedded into $G$ via the modified Cartan map $\mathfrak{C}_h$  such that  $\mathfrak{C}_h(G/\hat{K})$ is a connected component of
$\sqrt{e}$,\  where $\sqrt{e}:=\{g\in G| g^2=e \}.$}
 Let  $$ {\T}:\Omega G\rightarrow \Omega G,\ \  {\T}(\gamma)(\lambda)=\gamma(-\lambda)\gamma(-1)^{-1}$$
be an involution of $\Omega G$ (see Section 2.2) with the fixed point set
\[(\Omega G)_{ {\T}}=\{\gamma\in\Omega G|  {\T}(\gamma)=\gamma\}.\]
 { We give a version of Proposition 5.2 of \cite{BuGu} adapted to our approach.}

\begin{proposition}\label{eq-check-xi} { Let $\mathcal F$ be a harmonic map of finite uniton number with an extended frame $F(z,\bar z, \lambda)$. Let $\Phi(z,\bar z,1,\lambda)=A(1, {\lambda})^{-1} F(z,\bar{z}, \lambda)F(z,\bar{z}, 1)^{-1}$ be an extended solution for the harmonic map $\mathbb{F}=\mathfrak{C}_h\circ\mathcal F$ as described in
(\ref{eq-lawson2}) of Corollary \ref{cor-Phi-F}.} Then there exists some $\tilde\xi \in \mathfrak{t}$  satisfying
$\exp( \pi \tilde{\xi}) = h$, where $\mathfrak{t}$  denotes the Lie algebra of  {the maximal torus
$\mathrm{T}$ in $G$.}
 Setting $\gamma_{\tilde\xi} (\lambda) = \exp (t \tilde\xi)$ for $\lambda = e^{it}$ and $\tilde{\gamma} (\lambda) = \gamma_{\tilde\xi} (\lambda) A(1, \lambda)$ we obtain
$\tilde{\gamma}(\lambda)\in\Omega G$ , $\tilde{\gamma}(-1)=e$ and
\begin{equation}\label{eq-TP}
 {\T}(\tilde{\gamma}(\lambda)\Phi(z,\bar z,1,\lambda))=\tilde{\gamma}(\lambda)\Phi(z,\bar z,1,\lambda).
 \end{equation}
 Moreover,  $\tilde{\gamma}(\lambda)\Phi$ is also an extended solution for
 $\mathbb F$.
Furthermore, if $\Phi(z,\bar z,1,\lambda)$ takes  {values} in $\Omega _{alg}G$, then we also have  $\tilde{\gamma}(\lambda)\in \Omega _{alg}G$ and  $\tilde{\gamma}(\lambda)\Phi(z,\bar z,1,\lambda)$ takes  {values} in $(\Omega _{alg}G)_{ {\T}}$.
\end{proposition}

\begin{proof} The proof follows closely the one of \cite{BuGu}. First we note that $h$ is contained in the maximal torus. This implies the existence of $\tilde\xi$.
Next we verify \eqref{eq-TP}. Since \[\gamma_{\tilde\xi}(-\lambda)=\exp((\pi+t)\tilde \xi)=
\gamma_{\tilde\xi}(\lambda)h\ \hbox{ and }\ F(z,\bar{z}, -\lambda)h=hF(z,\bar{z}, \lambda),\] we have
\[\begin{split}
 {\T}(\tilde{\gamma}(\lambda)\Phi(z,\bar z,1,\lambda))
&= {\T}\left(\gamma_{\tilde\xi}(\lambda)F(z,\bar{z}, \lambda)F(z,\bar{z}, 1)^{-1}\right)\\
&=\gamma_{\tilde\xi}(-\lambda)F(z,\bar{z}, -\lambda)F(z,\bar{z}, 1)^{-1}F(z,\bar{z}, 1)F(z,\bar{z}, -1)^{-1}\gamma_{\tilde\xi}(-1)\\
&=\gamma_{\check\xi}(\lambda)hF(z,\bar{z}, -\lambda)F(z,\bar{z}, -1)^{-1}h\\
&=\gamma_{\check\xi}(\lambda)F(z,\bar{z}, \lambda)hF(z,\bar{z}, 1)^{-1}\\
&=\tilde{\gamma}(\lambda)\Phi(z,\bar z,1,\lambda).
\end{split}\]
\end{proof}
 {By Proposition \ref{eq-check-xi} we can assume without loss of generality that $\Phi$ has the form \begin{equation}\label{eq-Phi-F}
\Phi=\gamma_{\tilde\xi}(\lambda)F(z,\bar{z}, \lambda)F(z,\bar{z}, 1)^{-1}.\end{equation}}
With this $\Phi$ we obtain:

\begin{theorem}\label{thm-finite-uniton1}(Proposition 5.3, Theorem 5.4  {and p.567 of } \cite{BuGu})
Assume that $G$ is connected, compact, and semisimple with trivial center.

  \begin{enumerate}
  \item Let $\Phi:M\rightarrow (\Omega^k_{alg}G)_\T$ be an extended solution for some harmonic map
   { $\mathcal F$  of finite uniton number related by \eqref{eq-Phi-F}.}  Then there exists some canonical element $\xi$ in $\mathcal{I}'$, some $\gamma \in \Omega_{alg}G$  {satisfying $\gamma(-1) = e,$ }and some discrete subset $D'$ of $M$ such that on $M\setminus D'$,
\begin{equation}
  { \gamma\Phi =\exp C \cdot \gamma_{\xi} \cdot (\Phi^+_{\xi})^{-1},}
 \end{equation}
where  {$C:M\rightarrow (\mathfrak{u}^0_{\xi})_\T$ is a {\em meromorphic map} on $M$ with poles in $D'$} and $$(\mathfrak{u}^0_{\xi})_\T=\bigoplus_{0\leq 2j <r(\xi)}\lambda^{2j}(\mathfrak{f}^{\xi}_{2j})^{\perp}.$$
Moreover, $\xi$  satisfies \begin{equation}G/\hat{K} \cong\{g(\exp \pi \xi)g^{-1}| g\in G\}.
\end{equation}

 {Note that both $\Phi$ and $\gamma \Phi$ are extended solutions of $\mathcal F$.}
  \item
  Conversely, let  $\xi\in \mathcal{I}'$ be a canonical element.
    {Assume that $C:M \rightarrow (\mathfrak{u}^0_{\xi})_\T$ is a meromorphic function} such that
 {\begin{equation}
(\exp C)^{-1}\dd(\exp C)=\sum_{0\leq 2j\leq r(\xi)-1}\lambda^j A_{2j}'\dd z,\
\end{equation}
with $A_{2j}':M \rightarrow \mathfrak{g}^{\xi}_{2j+1}$ being meromorprhic,
$\ 0\leq 2j\leq r(\xi)-1$.}
Let $D'\subset M$ be the set of poles of $C$, then
 {$\Phi_{M\setminus D'}=\gamma_{\xi}^{-1} \cdot  \exp C  \cdot \gamma_{\xi} \cdot (\Phi^+_{\xi})^{-1} $} is an extended solution  { for some harmonic map of finite uniton number  $F:M\setminus D' \rightarrow G/\hat{K}\cong\{g(\exp \pi \xi)g^{-1}| g\in G\}\subset G$.}
 \end{enumerate}\end{theorem}

\begin{remark}  A   generalization of Burstall and Guest's theory for outer compact symmetric spaces has been published in \cite{Esch-Ma-Qu}.
\end{remark}

\subsection{The Burstall-Guest theory in relation to standard DPW theory} \label{BuGu<->DPW}


Using the above theorems, we can derive the normalized potential of harmonic maps of finite uniton type, showing that they are meromorphic 1-forms taking values in a fixed nilpotent Lie algebra.

  {Note that these results have been explained in Appendix B of \cite{BuGu} and Theorem 1.11 of \cite{Gu2002}. Here we rewrite them in terms of our language as well as a proof for later applications and for the convenience of readers. We also would like to point out that the work of Burstall and Guest does  not consider initial conditions, e.g.  for extended frames. So in general the value of an extended  frame will
 be different from $e$ at some fixed basepoint $z_0$.  We will show below in Theorem 4.20 below that also the standard DPW theory with prescribed initial condition at some fixed base point produces all harmonic maps of finite uniton number as well. All statements in the theorems below can be derived from the work of Burstall and Guest. But the presentation uses substantially  the DPW method. Therefore, for the convenience of the reader, we include a proof.}

\begin{theorem} \label{thm-finite-uniton2}We retain the  {the notation and the assumptions  of  Theorem \ref{thm-finite-uniton0} and Theorem \ref{thm-finite-uniton1} as needed.}
\begin{enumerate}
  \item
  Let $\mathbb{F}:M\rightarrow G$ be a harmonic map of finite uniton number  {with extended solution $\Phi(z,\bar z,\lambda)$ as stated in Theorem \ref{thm-finite-uniton0}. Then ${\Phi_-}:=\gamma_{\xi}^{-1}\cdot \exp C \cdot \gamma_{\xi}=\gamma_{\xi}^{-1}\cdot\Phi\cdot\Phi_+$} \violet{has} a Maurer-Cartan form
 {\begin{equation}\label{eq-eta} {\mathbb A_-}:= {\Phi_-}^{-1}\mathrm{d} {\Phi_-}=\lambda^{-1}\sum_{0\leq j\leq r(\xi)-1} A_j'\dd z, \end{equation}
where each $A_j':M\rightarrow \mathfrak{g}^{\xi}_{j+1}$ is a meromorphic function  {on $M$} with poles in $D'$.}
Moreover, at some base point $z_0\in M\setminus D'$ we have
 \begin{equation}\label{eq-initial1}
 {\Phi_-}(z_0)= {\Phi_{-0}}:= {\gamma_{\xi}^{-1} \cdot \exp C(z_0) \cdot \gamma_{\xi}} \in \Lambda^- G^{\C},\ C(z_0)\in\mathfrak{u}^0_{\xi}.
 \end{equation}

 Conversely, given $ {\mathbb A_-}$ which takes values in  $\lambda^{-1}\cdot\sum_{0\leq j\leq r(\xi)-1}\mathfrak{g}^{\xi}_{j+1}$ and an initial condition of $ {\Phi_-}$ of the form \eqref{eq-initial1},  {assume that  on $ M$,
there exists a global meromorphic solution $ {\Phi_-}$  (with poles in $\D'$)} \violet{to}
  \begin{equation}\label{eq-A-1} {\Phi_-}^{-1}\mathrm{d} {\Phi_-}= {\mathbb A_-}, \ ~~~ {\Phi_-}(z_0)= {\Phi_{-0}}= {\gamma_{\xi}^{-1} \cdot \exp C_0 \cdot \gamma_{\xi}\ } \in \Lambda^- G^{\C}\hbox{ with } C_0\in\mathfrak{u}^0_{\xi}.\end{equation}
  {The Iwasawa decomposition of $\gamma_{\xi}\Phi_-$ gives a harmonic map $\mathbb{F}$ of finite uniton number in $G$.}

  {The above two procedures are inverse to each other when the initial conditions match.}
 \item Let $\mathcal{F}:M\rightarrow G/\hat{K}$ be a harmonic map of finite uniton number
  {with extended frame $F(z,\bar z,\lambda)$ based at $z_0 \in M$ with initial value $e$.} Embed $G/\hat{K}$ into $G$ as totally geodesic submanifold via the  { modified} Cartan embedding. Then there exists some canonical $\xi\in \mathcal{I}'$,
some discrete subset $D'\subset M$, such that
 \begin{equation}
G/\hat{K}\cong\{g(\exp \pi \xi)g^{-1}| g\in G\},
\end{equation}
and that $ {F_-= \gamma_{\xi}^{-1} \cdot \exp C \cdot \gamma_{\xi}}$ is a meromorphic extended frame of $\mathcal{F}$ with the normalized potential having the form
 {\begin{equation}\eta=F_-^{-1}\mathrm{d}F_-=\lambda^{-1}\sum_{0\leq 2j\leq r(\xi)-1} A_{2j}'\dd z, \end{equation}
where $A_{2j}':M\rightarrow \mathfrak{g}^{\xi}_{2j+1}$ is a meromorphic function  on $M$ with poles in $D'$ for each $j$.}  And at the base point $z_0$
 \begin{equation}\label{eq-initial2}F_-(z_0)=F_{-0}:= {\gamma_{\xi}^{-1} \cdot \exp C(z_0) \cdot \gamma_{\xi} } \ in\ \Lambda^- G^{\C},\ C(z_0)\in(\mathfrak{u}^0_{\xi})_T.
 \end{equation}
 Conversely, given a  meromorphic normalized potential $\eta$ which takes values in  $\lambda^{-1}\cdot\sum_{0\leq 2j\leq r(\xi)-1}\mathfrak{g}^{\xi}_{2j+1}$ and an initial condition
  {$F_{-0}$ of $F_-$ of } the form \eqref{eq-initial2},
 {assume that on $M$ there exists a global  solution $F_-$ (with poles in $\D'$)}  {to}
 \begin{equation}\label{eq-initial3}F_-^{-1}\mathrm{d}F_-=\eta, \ ~~~F_-(z_0)=F_{-0}. \end{equation}
 {The Iwasawa decomposition of $F_-(z,\lambda)$ gives the extended frame of a harmonic} map of finite uniton number into $ \{g(\exp \pi \xi)g^{-1}| g\in G\}\cong G/\hat{K} $.

  {The above two procedures are inverse to each other when the initial conditions match.}
\end{enumerate}
\end{theorem}

 {Note that part (1) of Theorem \ref{thm-finite-uniton2} is essentially Proposition B1 of Appendix B of \cite{BuGu}.}

\begin{proof}
(1) First we note that by page 557 in \cite{BuGu},
\[ \gamma_{\xi}^{-1}X\gamma_{\xi}=\lambda^{-j-1}X,~ \hbox{ for any element } ~X\in\mathfrak{g}^{\xi}_{j+1}.\]
Together with the definition of $C$  { in Theorem \ref{thm-finite-uniton0},} we have
 { ${\gamma_{\xi}}^{-1} \cdot \exp C \cdot \gamma_{\xi} \in\Lambda^-G^{\mathbb{C}}$.}

Now consider the Maurer-Cartan form of  {$\gamma_{\xi}^{-1} \cdot \exp C \cdot \gamma_{\xi}$.} We have
\begin{equation*}\begin{split}( {\gamma_{\xi}^{-1} \cdot \exp C \cdot \gamma_{\xi})^{-1}\cdot }\mathrm{d}( {\gamma_{\xi}^{-1} \cdot \exp C \cdot \gamma_{\xi}})&=( {\gamma_{\xi}^{-1} \cdot
\exp C \cdot \gamma_{\xi}})^{-1}( {\gamma_{\xi}^{-1} \cdot \exp C \cdot \gamma_{\xi}})_z\mathrm{d}z \\
&=
 \gamma_{\xi}^{-1}\left( \exp C^{-1}(\exp C)_z\right) \gamma_{\xi}\mathrm{d}z,\end{split}\end{equation*}
since $\gamma_{\xi}$ is independent of $z$.

By \eqref{eq-C} in Theorem \ref{thm-finite-uniton0},
\begin{equation*}\label{eq-exp-C} (\exp C)^{-1}(\exp C)_z=\sum_{j=0}^{r(\xi)-1}\lambda^j A'_j \ \ \hbox{ with }\ \ \ A'_j:M \rightarrow \mathfrak{g}^{\xi}_{j+1},\ 0\leq j\leq r(\xi)-1.\end{equation*}
So
\begin{equation*}( {\gamma_{\xi}^{-1} \cdot \exp C \cdot \gamma_{\xi}})^{-1}
( {\gamma_{\xi}^{-1}\cdot \exp C \cdot \gamma_{\xi}})_z=\gamma_{\xi}^{-1} \left( \sum_{1\leq j\leq r(\xi)-1}\lambda^j A'_j\right)   \gamma_{\xi} =\sum_{1\leq j\leq r(\xi)-1} \lambda^j(\gamma_{\xi}^{-1} A'_j  \gamma_{\xi}).
\end{equation*}
By page 557 in \cite{BuGu}, for any element $X\in\mathfrak{g}^{\xi}_{j+1}$, $ \gamma_{\xi}^{-1}X\gamma_{\xi}=\lambda^{-j-1}X$.  Since $ A'_j $ takes values in $\mathfrak{g}^{\xi}_{j+1}$, we have
\[\gamma_{\xi}^{-1} A'_j  \gamma_{\xi}=\lambda^{-j-1}A'_j.\]
As a consequence,
\[( {\gamma_{\xi}^{-1}\cdot \exp C \cdot \gamma_{\xi}})^{-1}( {\gamma_{\xi}^{-1} \cdot
\exp C \cdot \gamma_{\xi}})_z=\sum_{1\leq j\leq r(\xi)-1} \lambda^j(\gamma_{\xi}^{-1} A'_j  \gamma_{\xi})= \lambda^{-1}\sum_{1\leq j\leq r(\xi)-1} A'_j .\]

 Conversely, let  $\Phi_-$ be a solution to \eqref{eq-A-1} as assumed. Then  $\Phi_-$ takes values in
$\Omega_{alg}G$, since $\lambda\mathbb A_-(\frac{\partial}{\partial z})$ takes values in a nilpotent Lie algebra.
{Letting $ \gamma_{\xi}\Phi_- =\Phi (\Phi_+)^{-1}$ } be the Iwasawa decomposition of $\gamma_{\xi}\Phi_-$. It hence produces a harmonic map of finite uniton number.

 (2) First one needs to restrict the above results to the case $A'_{2j+1}=0$ for all $j$ by  Theorem \ref{thm-finite-uniton1}.  By $(1)$ of Theorem \ref{thm-finite-uniton1} we can choose the extended solution associated to the harmonic map $\mathcal{F}$ such that
  $\Phi=\gamma_{\tilde\xi}(\lambda)F(z,\bar{z}, \lambda)h F(z,\bar{z}, 1)^{-1}$ holds. We see that in this case the Iwasawa decomposition of $\Phi$ yields
  $ \Phi_-=\Phi\Phi_+$, whence
   { \[\Phi_-=\Phi\Phi_+=\gamma_{\tilde\xi}(\lambda)F(z,\bar{z}, \lambda)hF(z,\bar{z}, 1)^{-1}\Phi_+,\]
 where the second equality follows  from (\ref{eq-Phi-F}).}
 {Consider $\tilde F(z,\bar z,\lambda):=\gamma_{\xi}(\lambda)^{-1}\gamma_{\tilde\xi}(\lambda)F(z,\bar{z}, \lambda)$. It is also an extended frame of the harmonic map $\mathcal F(z,\bar z,\lambda=1)$, but at the basepoint $z_0$ it might have an initial condition different from the initial condition of $F(z,\bar z, \lambda)$:  $\tilde F(z_0,\bar z_0,\lambda)=\gamma_{\xi}(\lambda)^{-1}\gamma_{\tilde\xi}(\lambda)F(z_0,\bar{z}_0, \lambda)$.} Also note that $\tilde F(z,\bar z,1)=F(z,\bar{z}, 1)$. We can w.l.g. replace $F(z,\bar z,\lambda)$ by $\tilde F(z,\bar z,\lambda)$, since we do not assume any special initial conditions. Then we have
\[\Phi_-=F(z,\bar{z}, \lambda)h F(z,\bar{z}, 1)^{-1}\Phi_+=F_-(z,\lambda)F_+(z,\bar{z}, \lambda)h F(z,\bar{z}, 1)^{-1}\Phi_+.\]
So $F_-(z,\lambda)=\Phi_-$. By (1) of this Theorem, we finish the proof of (2).
 \end{proof}

\begin{remark}\
\begin{enumerate}
\item Note that $\mathbb{A}_-$ in \eqref{eq-initial1} is not the usual normalized potential, since the harmonic map is mapped into the Lie group $G$ instead of $G/K$ and we do not have the initial condition $e$. Also by the discussion in Section 4.1.1 and for convenience we still call it ``normalized potential". For more relations between $\mathbb{A}_-$ and the real normalized potential (embedding $G$ to construct a harmonic map into $G\times G/G$), we also refer the readers to \cite{Do-Es}.

\item The initial condition  {$\Phi_{-0}$ and $F_{-0}$ respectively} of Theorem \ref{thm-finite-uniton2} can be removed by using dressing (see Theorem 1.11 of \cite{Gu2002}). For instance, assume that
$\hat{F}_-^{-1}\mathrm{d}\hat{F}_-=\eta, \  {\hat{F}_-(z_0,\lambda)}=e$. Then $F_-=F_{-0}\hat{F}_-$. By Iwasawa splitting we have
\[F_-=FF_+,\ \   \hat{F}_-=\hat{F}\hat{F}_+.\ \]
Assume that $F_{-0}=\gamma_0\gamma_+$ with $\gamma_0\in\Lambda G,$ $\gamma_+\in\Lambda^+ G^{\mathbb{C}}$. Therefore we have
\[FF_+=\gamma_0\gamma_+\hat{F}\hat{F}_+.\]
As a consequence, we obtain
\begin{equation*} \gamma_+\hat{F}=\gamma_0^{-1}FF_+\hat{F}_+^{-1}.\end{equation*}
Hence
\begin{equation}\gamma_+\sharp\hat{F}=\gamma_0^{-1}F.\end{equation}
So up to a rigid motion $\gamma_0^{-1}$, $F$ is the dressing of $\hat{F}$ by $\gamma_+$ (compare with Corollary \ref{cor-finite-dress}).
\end{enumerate}
  \end{remark}

\subsection{On the initial conditions of normalized meromorphic frames: compact case}

In the last theorem we have given a precise relation between certain ``meromorphic potentials" and harmonic maps of finite uniton type. For this one needs very specific initial conditions and  { very specific dressing matrices} respectively.
 { In the case of the last theorem the initial conditions occurring are very  complicated and very difficult to produce. So it is important and useful to show that all harmonic maps of finite uniton type into compact inner symmetric spaces can be determined by the data as given, but with initial condition $e$, which is the goal of this subsection}.

To this end, we need some preparations.
First, for an arbitrary element  {$\xi \in \mathcal{I}'$} we have the following decompositions
 {(see e.g. \eqref{defgjxi}):}
\begin{equation}\label{eq-pq-decom-g}
  { \mathfrak{g}^{\C}=\sum_j  \mathfrak{g}^{\xi}_j=\left(\sum_{j\geq0} \mathfrak{g}^{\xi}_j\right)\oplus\left(\sum_{j<0} \mathfrak{g}^{\xi}_j\right)=\mathfrak{pr}\oplus\mathfrak{q}.}
\end{equation}
On the Lie group level, let $\mathds{PR}$ be the  {complex connected} Lie subgroup of $G^{\C}$ with Lie algebra
$\mathfrak{pr}$ and  let $\mathds{Q}$ denote the  {complex connected} Lie subgroup of $G^{\C}$ with Lie algebra $\mathfrak{q}$.
Let $\mathfrak{W}$ denote the Weyl group of $G$. Then we have the decomposition
\begin{equation}\label{eq-pq-decom-G}
    G^{\C}=\bigcup_{\omega\in \mathfrak{W_{\xi}}}\mathds{PR}\cdot\omega\cdot\mathds{Q}.
\end{equation}
where $\mathfrak{W_{\xi}}$ is  {some quotient of $\mathfrak{W}$.} This follows from Corollary 3.2.3 of \cite{Do-Gr-Sz}.

\begin{theorem} \label{thm-PR-Q-decomposition}
 {The set $\mathds{PR} \cdot \mathds{Q} \subset G^{\C}$ obtained by pointwise multiplication is open in $G^\C$ and the pointwise multiplication map $ \mathds{PR}\times\mathds{Q} \longrightarrow  \mathds{PR}\cdot\mathds{Q}$ is biholomorphic.}
 \end{theorem}

\begin{proof} By Theorem 2.4.1 (a) of \cite{Do-Gr-Sz}, $ \mathds{PR}\cdot\mathds{Q}$ is open in $G^{\C}$. By Theorem 2.4.1 (b) of \cite{Do-Gr-Sz},
  $ \mathds{PR}\times\mathds{Q}\longrightarrow  \mathds{PR}\cdot\mathds{Q}$ is a holomorphic diffeomorphism.
\end{proof}

We also need the following  { lemma:}
 \begin{lemma} \label{lemma-}  {Let $\xi \in \mathcal{I}'$ so that $h=\exp(\pi\xi)$. Then we obtain}
 \begin{equation}\label{eq-pr-p}
    \mathfrak{pr}\cap \mathfrak{p}^{\C}= \sum_{j\geq0}g^{\xi}_{2j+1}.
 \end{equation}
 \end{lemma}

\begin{proof}
 {Recall that $\mathfrak{g}=\mathfrak{k}\oplus\mathfrak{p}$
	with $\mathfrak{k}=Lie(K)$, and $\mathfrak{g}^{\C}=\mathfrak{k}^{\C}\oplus\mathfrak{p}^{\C}$.}
We have that $\mathfrak{k}=\{~X\in\mathfrak{g}~|~hX=Xh\ \},~~\mathfrak{k}^{\C}=\{~X\in\mathfrak{g}^{\C}~|~hX=Xh\ \}.$
For any element $X$ in $g^{\xi}_{2j}$,
\[hXh^{-1}=\exp( \pi\xi)\cdot X\cdot\exp(-\pi\xi)=\exp(\pi \cdot ad\xi) X=e^{2 j\pi  \sqrt{-1} }X=X.\]
Similarly, for any element $X$ in $g^{\xi}_{2j+1}$,
\[hXh^{-1}=\exp( \pi\xi)\cdot X\cdot\exp(-\pi\xi)=\exp( \pi\cdot ad\xi) X=e^{(2j+1)\pi \sqrt{-1}}X=-X.\]
\end{proof}

 {Next we will apply Theorem \ref{thm-finite-uniton2} and Theorem \ref{thm-PR-Q-decomposition} to show that for a harmonic map $\mathcal F$ an extended frame with initial condition $e$, Theorem \ref{thm-finite-uniton2} still holds.}

 \begin{theorem} \label{thm-finite-uniton-in}
Let G be a connected, semisimple, compact Lie group with trivial center. Let $\mathcal{F}:M\rightarrow G/\hat{K}$ be a harmonic map of
 finite uniton type  into the compact inner symmetric space $G/\hat{K}$.  {Let $z_0 \in M$ and $F(z, \bar z, \lambda)$ an extended frame of $\mathcal{F}$ satisfying
  $F(z_0, \bar z_0, \lambda) = e$ for all $\lambda \in S^1$.  Then there exists some canonical $\xi_{can}\in \mathcal{I}'$,
some discrete subset $D'\subset M$, such that
$\mathfrak{C}_h(G/\hat{K})=\{ghg^{-1}| g\in G\}$ with $h=\exp (\pi \xi_{can})$,
and  the normalized potential of $\mathcal{F}$ has the form
\begin{equation}\eta=F_-^{-1}\mathrm{d}F_-=\lambda^{-1}\sum_{0\leq 2j\leq r(\xi_{can})-1} A_{2j}'\mathrm{d}z, \end{equation}}
 {where $A_{2j}':M\rightarrow \mathfrak{g}^{\xi_{can}}_{2j+1}$ is a meromorphic function with poles in $D'$ for each $j$}. In particular, the normalized potential of $ \mathcal{F}$ is contained in a nilpotent {Lie algebra
 and  it is invariant} under the fundamental group of $M$.
 \end{theorem}
 \begin{proof} {
 By  Proposition \ref{eq-check-xi}, there exists some $\tilde\xi\in\mathcal I'$ such that
$\Phi=\gamma_{\tilde\xi} F(z,\bar{z},  \lambda)F(z,\bar{z}, 1)^{-1}$ is  an extended solution  of $\mathbb F(z,\bar z,1)=F(z, \bar z, 1)hF(z, \bar z, 1)^{-1}$  which is ${\T}-$invariant and satisfies $\exp(\pi\tilde\xi)=h$. Then by Proposition 4.1 of \cite{BuGu}, there exists some $\hat\xi\in\mathcal I'$ which may not be canonical, and some discrete subset $D'\subset M$ such that on $M\backslash D'$ we have
\[\gamma^{-1}_{\hat\xi}\Phi=\gamma^{-1}_{\hat\xi} \cdot \exp C \cdot \gamma_{\hat\xi}\cdot W_+(z,\bar{z},\lambda)=\hat{F}_- (z,\lambda) W_+(z,\bar{z},\lambda), \text{where } \hat{F}_-(z,\lambda)= \gamma^{-1}_{\hat\xi}
\cdot \exp C \cdot \gamma_{\hat\xi},\]
with $C:M\rightarrow(\mathfrak{u}^0_{\hat\xi})_{\T}$ being meromorphic with poles in $D'$ and
{$W_+: M \rightarrow \Lambda^{+}G^{\C}$.} Note that {$\gamma_{\hat\xi}\in(\Omega_{alg}G)_{\T}$}
since both $\Phi$ and $\gamma_{\tilde\xi} F(z,\bar{z},  \lambda)F(z,\bar{z}, 1)^{-1}$ are contained in $(\Omega_{alg}G)_{\T}$. In particular, $\exp(\pi\hat\xi)=h$ and both of $\gamma_{\tilde\xi}^{-1}$ and $\gamma_{\hat\xi}$  take values in $\mathds{PR}$,  {where we define $\mathds{PR}$ by $\hat\xi$}. }

 {Let \[\hat{F}(z,\bar{z},\lambda)=\gamma_{\hat\xi}^{-1}\gamma_{\tilde\xi}F(z,\bar{z},\lambda).\]
Since all of $\gamma_{\xi}$, $\gamma_{\tilde\xi}$ and $F(z,\bar{z},\lambda)$ are $\sigma-$twisted, we see that $\hat{F}(z,\lambda)$ is $\sigma-$twisted and hence is an extended frame of $\mathcal F$ with different initial condition  {$\hat{F}(z_0,\bar{z}_0,\lambda)=\gamma_{\hat\xi}(\lambda)^{-1}\gamma_{\tilde\xi}(\lambda)$.}
Moreover, since $ \hat{F}_-(z,\lambda)=  {\gamma^{-1}_{\hat\xi} \cdot \exp C \cdot \gamma_{\hat\xi}}$ takes values in $\Lambda^-G^{\mathbb C}_{\sigma}$, we have the Birkhoff decomposition of $\hat F(z,\bar{z},\lambda)$:
\[\hat{F}=\gamma^{-1}_{\xi}\Phi(z,\bar{z},\lambda) F(z,\bar{z},1)=\hat{F}_-(z,\lambda) \hat{ V}_+(z,\bar{z},\lambda) \ \hbox{ with }\ \hat{ V}_+(z,\bar{z},\lambda)=W_+(z,\bar{z},\lambda)F(z,\bar{z},1).\]
Now near $z_0$, we also have the Birkhoff decomposition of $F(z,\bar{z},\lambda)$:
\[F(z,\bar{z},\lambda)=F_-(z,\lambda)V_+
=\gamma_{\tilde\xi}^{-1}\gamma_{\hat\xi}\hat{F}(z,\bar{z},\lambda)
=\gamma_{\tilde\xi}^{-1}\gamma_{\hat\xi}\hat{F}_-(z,\lambda)\hat{ V}_+(z,\bar{z},\lambda).\]
Decompose $\hat{ V}_+(z,\bar{z},\lambda)$ according to \eqref{eq-pq-decom-G}:
\[\hat{ V}_+(z,\bar{z},\lambda)=\mathbf{R}\omega \mathbf{Q}~~ \hbox{ with }~ \mathbf{R}\in\mathds{PR} \hbox{ and } \mathbf{Q}\in\mathds{Q}.\]
By Theorem \ref{thm-PR-Q-decomposition}, since $\hat V_+$ is holomorphic in $\lambda$ for all $\lambda\in\mathbb C$, $\mathbf{R}$ and  $\mathbf{Q}$ are also holomorphic in $\lambda$ for all $\lambda\in\mathbb C$, i.e., $\mathbf{R}, \mathbf{Q}\in\Lambda^+G^{\C}$.}

 {Since  $F(z,\bar{z},\lambda)\rightarrow e$ as $z\rightarrow z_0$, we obtain \[\gamma_{\tilde\xi}^{-1}\gamma_{\hat\xi}\hat{F}_-(z,\lambda) \mathbf{R}\omega \mathbf{Q}\rightarrow e, \hbox{ if $z\rightarrow z_0$}.\] Note that all of
$\gamma_{\tilde\xi}^{-1}$, $\gamma_{\hat\xi}$ and $\hat{F}_-(z,\lambda)$ take values in $\mathds{PR}$. Since $\mathds{PR}\cdot\mathds{Q}$ is open and $e\in\mathds{PR}\cdot\mathds{Q}$, $\omega=e$  near $z_0$. As a consequence, when $z\rightarrow z_0$,  $\gamma_{\tilde\xi}^{-1}\gamma_{\hat\xi}\hat{F}_-(z,\lambda) \mathbf{R}\rightarrow e$ and $\mathbf{Q}\rightarrow e$. Moreover, considering $F_-$, we obtain
\[F_-=(\gamma_{\tilde\xi}^{-1}\gamma_{\hat\xi}\hat{F}_-(z,\lambda) \mathbf{R})_-.\]
Since $\mathbf{R}\in\mathds {PR}$ and all of
$\gamma_{\tilde\xi}^{-1}$, $\gamma_{\hat\xi}$ and $\hat{F}_-(z,\lambda)$ also take values in $\mathds{PR}$, we see that $F_-$ also takes values in $\mathds{PR}$.
Consider \[\eta_-=F_-^{-1}\mathrm{d}F_-=\lambda^{-1}\eta_{-1}\mathrm{d}z.\] Since $F_-$ takes values in $\mathds{PR}$, $\eta_{-1}$ takes values in $\mathfrak{pr}$. On the other hand, $\eta_{-1}$ also takes values in $\mathfrak{p}^{\C}$. In a sum, by \eqref{eq-pr-p}, $\eta_{-1}$ takes values in \[\mathfrak{pr}\cap\mathfrak{p}^{\C}=\sum_{0\leq 2j\leq r(\hat\xi)-1} \mathfrak{g}^{\hat\xi}_{2j+1}.\]  Let $\xi_{can}$ be the canonical element derived from $\hat\xi$ as in Lemma \ref{lemma-xi}. By Lemma \ref{lemma-xi} we have
\[\sum_{0\leq j\leq r(\hat\xi)-1} \mathfrak{g}^{\hat\xi}_{j+1}=\sum_{0\leq j\leq r(\xi_{can})-1}  \mathfrak{g}^{\xi_{can}}_{j+1}.\]
From the proof of Theorem 5.4 of \cite{BuGu}, we also obtain
\[ \sum_{0\leq 2j\leq r(\hat\xi)-1} \mathfrak{g}^{\hat\xi}_{2j+1}=\sum_{0\leq 2j\leq r(\xi_{can})-1}  \mathfrak{g}^{\xi_{can}}_{2j+1}.\]
 So   $\eta_{-1}$ takes values in $\sum_{0\leq 2j\leq r(\xi_{can})-1}  \mathfrak{g}^{\xi_{can}}_{2j+1}$, which is contained in the same nilpotential Lie subalgebra as $\hat\eta_{-1}=\hat F_{-}^{-1}(\hat F_{-})_{z}$. Since $F(z_0,\bar{z}_0,\lambda)=e$, we have $F_{-}(z_0,\lambda)=e$ as well. }
  \end{proof}

\section{Harmonic maps of finite uniton type into non-compact inner symmetric spaces}

 {In the last section we have shown how one can relate the construction schemes of Burstall-Guest and DPW respectively of harmonic maps into compact inner symmetric spaces into each other. In this section, we will show that the DPW interpretation of the theory of Burstall and Guest \cite{BuGu} also holds for harmonic maps of finite uniton type into non-compact inner symmetric spaces. We will also review briefly the application of this theory to the coarse} classification of Willmore surfaces \cite{Wang-2}.

\subsection{ {The case of non-compact inner symmetric spaces}}
 {We will apply the results of  Section 3.3. Let $G/K$ is a non-compact inner symmetric space} and $U$  a maximal compact subgroup of the complexification $G^\C$ of $G$ which is left invariant by the natural real form involution of $G$ and the extension of $\sigma$ to $G^\C$. Combining Theorem \ref{thm-noncompact}, Theorem \ref{thm-finite-uniton2},   {Theorem \ref{thm-finite-uniton-in}} and Corollary \ref{cor-finite}, we obtain
\begin{theorem}\label{thm-finite-uniton-n-com}
Let $\mathcal{F}: \tilde{M}\rightarrow G/K$  be a harmonic map of finite uniton type, and
 ${\mathcal{F}_ {U}:} \tilde{M} \rightarrow U/(U\cap K^{\C})$ the compact dual harmonic map of $\mathcal{F}$, with base point $z_0\in\tilde{M}$  {such that $\mathcal{F}_{z=z_0}=eK$ and $(\mathcal{F}_{U})_{z=z_0}=e(U\cap K^{\C})$.}  {Then the normalized potential $\eta_-$ and the normalized extended framing  $F_-$ derived from  Theorem \ref{thm-finite-uniton-in} for
 $\mathcal{F}_U$,  provides also a normalized potential $\eta_-$ and a normalized extended framing $F_-$  of $\mathcal{F}$, with initial condition $F_-(z_0,\lambda)=e$.} {In particular, all harmonic maps of finite uniton type $\mathcal{F}: \tilde{M}\rightarrow G/K$  can be obtained in this way.}
\end{theorem}

 {\begin{remark}\
\begin{enumerate}
\item By Theorem \ref{thm-finite-uniton-n-com}, on the (nilpotent) normalized potential level, one can classify  all harmonic maps $\mathcal{F}: \tilde{M}\rightarrow G/K$ of finite uniton type by classifying  all harmonic maps of finite uniton type from $\tilde{M}$ to $U/(U\cap K^{\C})$.
\item Note that the DPW method yields a 1-1 relation between normalized potentials
and harmonic maps if one chooses a base point and an initial condition $e$ for all relevant extended frames.
\end{enumerate}
\end{remark}}
\subsection{Nilpotent normalized potentials of Willmore surfaces of finite uniton type}

We will end this paper with a  {brief view} of a coarse classification and a construction of new Willmore surfaces of finite uniton type, in the spirit of Section 4. To be concrete, we  give a coarse classification of Willmore two-spheres in $S^{n+2}$ by classifying all the possible nilpotent Lie sub-algebras related to the corresponding harmonic conformal Gauss maps, see Theorem 3.1, Theorem 3.3 of \cite{Wang-1}. This classification indicates that Willmore two spheres may inherit more complicated and new geometric structures.  Moreover, by concrete computations of Iwasawa decompositions, we construct new Willmore two-spheres (Section 3.1 and Section 3.2 of \cite{Wang-3}).

To state the coarse classification and constructions, we first recall that a Willmore surface $y:M\rightarrow S^{n+2}$ is globally related to a harmonic conformal Gauss map  $\mathcal{F}:M\rightarrow SO^+(1,n+3)/SO^+(1,3)\times SO(n)$ satisfying some isotropy condition (\cite{Bryant1984}, \cite{Ejiri1988}, \cite{DoWa11}, \cite{Wang-1}). Here
$SO^+(1,n+3)=SO(1,n+3)_0$ is the connected subgroup of
\begin{equation*}SO(1,n+3):=\{A\in Mat(n+4,\mathbb{C})\ |\ A^tI_{1,n+3}A=I_{1,n+3}, \det A=1,\ A=\bar{A}
\},
\end{equation*}
and the subgroup $K=SO^+(1,3)\times SO(n)\subset SO^+(1,n+3)$ is defined by the involution
 \begin{equation}\begin{array}{ll}
\sigma:    SO^+(1,n+3)&\rightarrow SO^+(1,n+3)\\
 \ \ \ \ \ \ \ A &\mapsto DAD^{-1}.
\end{array}\end{equation}
with $D=diag\{-I_4,I_n\}.$ Moreover
\begin{equation*} \mathfrak{so}(1,n+3):=\{A\in Mat(n+4,\mathbb{R})\ |\ A^tI_{1,n+3}+I_{1,n+3}A\}.
\end{equation*}
Let $\mathfrak{k}$ be the Lie algebra of $SO^+(1,3)\times SO(n)\subset SO^+(1,n+3)$ and $\mathfrak{so}(1,n+3)=\mathfrak{k}\oplus\mathfrak{p}$. Then
\[\mathfrak{k}=\left\{\left(
                               \begin{array}{cc}
                                 A_1 & 0 \\
                                 0 & A_2 \\
                               \end{array}
                             \right)| A_1\in Mat(4,\mathbb{R}),\ A_2\in Mat(n,\mathbb{R}),\ A_1^tI_{1,3}+I_{1,3}A_1=0,\ A_2^t+A_2=0
\right\}\]
and
\[\mathfrak{p}=\left\{\left(
                               \begin{array}{cc}
                                0 & B_1 \\
                                 -B_1^tI_{1,3} &0 \\
                               \end{array}
                             \right)| B_1\in Mat(4\times n,\mathbb{R})
\right\}.\]
 {It is straightforward to see that the compact dual of $SO^+(1,n+3)$ and $SO^+(1,3)\times SO(n)$ are $SO(n+4)$ and $SO(4)\times SO(n)$ respectively.}

 Following \cite{DoWa11} that we call a conformally harmonic map  $\mathcal{F}:M\rightarrow SO^+(1,n+3)/SO^+(1,3)\times SO(n)$ a strongly conformally \footnote{ {Note that in \cite{BW} the notion of ``strongly conformal" has been used in another sense.}} harmonic map if the $\mathfrak{p}^{\C}-$part of its Maurer-Cartan form,
\[\left(
                     \begin{array}{cc}
                       0 & {B}_1 \\
                       -{B}^{t}_1I_{1,3} & 0 \\
                     \end{array}
                   \right)\mathrm{d}z,\]
                   satisfies the isotropy condition
                   \begin{equation}\label{eq-isotropic}
                    B_1^tI_{1,3}B_1=0.
                   \end{equation}

                   Now assume $n+4=2m$ so that $SO^+(2m)/SO^+(1,3)\times SO(2m)$ is an inner symmetric space.
\begin{theorem}(Theorem 3.1 of \cite{Wang-1})
Let $\mathcal{F}:M\rightarrow SO^+(1,n+3)/SO^+(1,3)\times SO(n)$ be a finite-uniton type strongly conformally harmonic map, with $n+4=2m$. Let
$$\eta=\lambda^{-1}\left(
                     \begin{array}{cc}
                       0 & \hat{B}_1 \\
                       -\hat{B}^{t}_1I_{1,3} & 0 \\
                     \end{array}
                   \right)\mathrm{d}z $$
  be the normalized potential of $\mathcal{F}$, with $$\hat{B}_1=(\hat{B}_{11},\cdots,\hat{B}_{1,m-2}),\ \hat{B}_{1j}=(\mathrm{v}_{j },\hat{ \mathrm{v}}_{j })\in Mat(4\times 2,\mathbb{C}).$$
 Then up to some conjugation by a constant matrix, every $\hat{B}_{j1}$ of $\hat{B}_1$ has one of the two forms:
\begin{equation}(i) \ \mathrm{v}_{j}= \left(
                                          \begin{array}{ccccccc}
                                            h_{1j}    \\
                                            h_{1j}  \\
                                            h_{3j} \\
                                            ih_{3j}  \\
                                          \end{array}
                                        \right),\ \hat{ \mathrm{v}}_{j } =
                                        \left(
                                          \begin{array}{ccccccc}
                                            \hat{h}_{1j}    \\
                                            \hat{h}_{1j}   \\
                                            \hat{h}_{3j} \\
                                            i\hat{h}_{3j}   \\
                                          \end{array}
                                        \right);\
 (ii) \ \mathrm{v}_{j}= \left(
                                          \begin{array}{ccccccc}
                                            h_{1j}  \\
                                            h_{2j}    \\
                                            h_{3j}  \\
                                            h_{4j}  \\
                                          \end{array}
                                        \right),\ \ \hat{ \mathrm{v}}_{j } = i \mathrm{v}_{j}= \left(
                                          \begin{array}{ccccccc}
                                            ih_{1j}  \\
                                            ih_{2j}    \\
                                            ih_{3j}  \\
                                            ih_{4j}  \\
                                          \end{array}
                                        \right).
\end{equation}
And all of $\{\mathrm{v}_j,\  \hat{ \mathrm{v}}_{j }\}$ satisfy the equations
$$\mathrm{v}_j^tI_{1,3}\mathrm{v}_l=\mathrm{v}_j^tI_{1,3}\hat{\mathrm{v}}_l=\hat{\mathrm{v}}_j^tI_{1,3}\hat{\mathrm{v}}_l=0, \ j,l=1,\cdots,m-2.$$
In other words, there are $m-1$ types of normalized potentials with $\hat{B}_{1}$ satisfying $\hat{B}_{1 }^tI_{1,3}\hat{B}_{1 }=0$, namely those being of one of the following $m-1$ forms (up to some conjugation):

$(1)$ (all pairs are of type (i))

\begin{equation}
                         \hat{B}_{1}= \left(
                                          \begin{array}{ccccccc}
                                            h_{11} & \hat{h}_{11} &  h_{12} & \hat{h}_{12} &\cdots &  h_{1,m-2}& \hat{h}_{1,m-2} \\
                                             h_{11} & \hat{h}_{11} &  h_{12}& \hat{h}_{12}&\cdots &  h_{1,m-2} & \hat{h}_{1,m-2}\\
                                            h_{31}& \hat{h}_{31} &  h_{32}& \hat{h}_{32} &\cdots &  h_{3,m-2}& \hat{h}_{3,m-2} \\
                                            ih_{31}& i\hat{h}_{31} &  ih_{32}& i\hat{h}_{32}&\cdots &  ih_{3,m-2}& i\hat{h}_{3,m-2} \\
                                          \end{array}
                                        \right);\end{equation}

$(2)$ (the first pair is of type (ii), all others are of type (i))
\begin{equation}
                           \hat{B}_{1}= \left(
                                          \begin{array}{ccccccc}
                                            h_{11} & i {h}_{11} &  h_{12} & \hat{h}_{12} &\cdots &  h_{1,m-2}& \hat{h}_{1,m-2} \\
                                             h_{21} & i {h}_{21} &  h_{12}& \hat{h}_{12}&\cdots &  h_{1,m-2} & \hat{h}_{1,m-2}\\
                                            h_{31}& i {h}_{31} &  h_{32}& \hat{h}_{32} &\cdots &  h_{3,m-2}& \hat{h}_{3,m-2} \\
                                            h_{41}& i {h}_{41} &  ih_{32}& i\hat{h}_{32}&\cdots &  ih_{3,m-2}& i\hat{h}_{3,m-2} \\
                                          \end{array}
                                        \right);\end{equation}
Introducing consecutively  more pairs of type (ii), one finally arrives at

       $(m-1)$ (all pairs are of type (ii))
\begin{equation}
                           \hat{B}_{1}= \left(
                                          \begin{array}{ccccccc}
                                            h_{11} & i {h}_{11} &  h_{12} & i {h}_{12} &\cdots &  h_{1,m-2}& i{h}_{1,m-2} \\
                                             h_{21} & i {h}_{21} &  h_{22}& i {h}_{22}&\cdots &  h_{2,m-2} & i{h}_{2,m-2}\\
                                            h_{31}& i {h}_{31} &  h_{32} & i {h}_{32} &\cdots &  h_{3,m-2}& i{h}_{3,m-2} \\
                                            h_{41}& i {h}_{41} &  h_{42} & i {h}_{42}&\cdots &   h_{4,m-2}& ih_{4,m-2} \\
                                          \end{array}
                                        \right).\end{equation}
                            {Here $r(f)\leq 2$ for Case $(1)$ and Case $(m-1)$. For other cases, $r(f)\leq 3,4,5,6$ or $8$, depending on the structure of the potential.}

\end{theorem}
 {This theorem follows from the classification of nilpotent Lie sub-algebras related to the symmetric space  $SO^+(1,n+3)/SO^+(1,3)\times SO(n)$, together with a restriction  of the isotropy condition \eqref{eq-isotropic} on potentials (to derive Willmore surfaces). See \cite{Wang-1} for more details.}
\begin{example}(\cite{DoWa12}, \cite{Wang-3})\label{example}
 Let \[\eta=\lambda^{-1}\left(
                      \begin{array}{cc}
                        0 & \hat{B}_1 \\
                        -\hat{B}_1^tI_{1,3} & 0 \\
                      \end{array}
                    \right)\mathrm{d}z,\ ~ \hbox{ with } ~\ \hat{B}_1=\frac{1}{2}\left(
                     \begin{array}{cccc}
                       2iz&  -2z & -i & 1 \\
                       -2iz&  2z & -i & 1 \\
                       -2 & -2i & -z & -iz  \\
                       2i & -2 & -iz & z  \\
                     \end{array}
                   \right).\]
 {Each extended frame $F(z,\bar z,\lambda)$ derived  from this potential has singularities, while the corresponding harmonic maps and  the corresponding  Willmore two-spheres are globally well-defined.}
The associated family of Willmore two-spheres $x_{\lambda}$, $\lambda\in S^1$, corresponding to $\eta$, is
\begin{equation}\label{example1}
\begin{split}x_{\lambda}&=\frac{1}{ \left(1+r^2+\frac{5r^4}{4}+\frac{4r^6}{9}+\frac{r^8}{36}\right)}
\left(
                          \begin{array}{c}
                            \left(1-r^2-\frac{3r^4}{4}+\frac{4r^6}{9}-\frac{r^8}{36}\right) \\
                            -i\left(z- \bar{z})(1+\frac{r^6}{9})\right) \\
                            \left(z+\bar{z})(1+\frac{r^6}{9})\right) \\
                            -i\left((\lambda^{-1}z^2-\lambda \bar{z}^2)(1-\frac{r^4}{12})\right) \\
                            \left((\lambda^{-1}z^2+\lambda \bar{z}^2)(1-\frac{r^4}{12})\right) \\
                            -i\frac{r^2}{2}(\lambda^{-1}z-\lambda \bar{z})(1+\frac{4r^2}{3}) \\
                            \frac{r^2}{2} (\lambda^{-1}z+\lambda \bar{z})(1+\frac{4r^2}{3})  \\
                          \end{array}
                        \right)\\
  \end{split}
\end{equation}
 $x_{\lambda}:S^2\rightarrow S^6$ is a Willmore immersion in $S^6$, which is non S-Willmore, full, and totally isotropic. In particular, $x_\lambda$ does not have any branch points. The uniton number of $x_{\lambda}$ is $2$ and therefore its conformal Gauss map is $S^1-$invariant by Corollary 5.6 of \cite{BuGu}.
\end{example}

{  \bf Acknowledgements}\ \
This work was started when  {PW} visited the Department of Mathematics of TU  M\"{u}nchen, and  the Department of Mathematics of Tuebingen University.  The paper was continued and finished during several mutual visits of the authors.
They would like to express their sincere gratitude for both the hospitality and the financial support.  {PW}  is thankful to Professor Changping Wang and Xiang Ma for their suggestions and encouragement.  {PW} was partly supported by the Project 11971107 of NSFC. PW is also thankful to the ERASMUS MUNDUS TANDEM Project for the financial supports to visit the TU M\"{u}nchen.

{\footnotesize
\def\refname{References}

}
{\small\

 Josef F. Dorfmeister

Fakult\" at f\" ur Mathematik, TU-M\" unchen,

Boltzmann str.3, D-85747, Garching, Germany

{\em E-mail address}: dorfm@ma.tum.de\\

Peng Wang

School of Mathematics and Statistics, FJKLMAA,

Fujian Normal University, Qishan Campus,

Fuzhou 350117, P. R. China

{\em E-mail address}: {pengwang@fjnu.edu.cn}

\end{document}